\documentclass{amsart}
\usepackage{amsmath,amscd,xypic,amssymb,combelow,color,enumitem,graphicx,float}

\usepackage[normalem]{ulem}
\usepackage[dvipsnames]{xcolor}
\makeatletter
\def\squiggly{\bgroup \markoverwith{\textcolor{red}{\lower3.5\p@\hbox{\sixly \char58}}}\ULon}
\makeatother

\xyoption{all}
\CompileMatrices

\makeatletter
\@addtoreset{equation}{section}
\makeatother

\newtheorem{theorem}[subsection]{Theorem}

\newtheorem{proposition}[subsection]{Proposition}

\newtheorem{definition}[subsection]{Definition}

\newtheorem{claim}[subsection]{Claim}

\newtheorem{remark}[subsection]{Remark}

\def\loccit{\emph{loc. cit. }}

\def\fsl{{\mathfrak{sl}}}
\def\fgl{{\mathfrak{gl}}}

\def\res{\textrm{Res}}

\def\BN{{\mathbb{N}}}

\def\BQ{{\mathbb{Q}}}
\def\BZ{{\mathbb{Z}}}

\def\CA{{\mathcal{A}}}
\def\CB{{\mathcal{B}}}

\def\CM{{\mathcal{M}}}

\def\CP{{\mathcal{P}}}

\def\CS{{\mathcal{S}}}

\def\tCA{\widetilde{\CA}}

\def\tR{\widetilde{R}}

\def\wCA{\widehat{\CA}}
\def\wCS{\widehat{\CS}}

\def\ph{\varphi}

\def\e{\varepsilon}

\def\vs{\varsigma}

\def\and{\textrm{ }\&\textrm{ }}

\def\sym{\textrm{Sym}}

\def\esym{\emph{Sym}}

\def\eop{\emph{op}}

\def\sym{\textrm{Sym}}

\def\nn{{{\BN}}^n}
\def\zz{{{{\mathbb{Z}}}^n}}

\def\uui{{U_q(\dot{\fgl}_1)}}

\def\uu{{U_q(\dot{\fgl}_n)}}

\def\uum{{U_q^-(\dot{\fgl}_n)}}
\def\uug{{U_q^\geq(\dot{\fgl}_n)}}
\def\uul{{U_q^\leq(\dot{\fgl}_n)}}

\def\UU{{U_{q,\oq}(\ddot{\fgl}_n)}}

\def\wuu{{\widehat{U}_{q}(\dot{\fgl}_n)}}

\def\wUU{{\widehat{U}_{q,\oq}(\ddot{\fgl}_n)}}

\def\bd{{\mathbf{d}}}

\def\br{{\mathbf{r}}}

\def\bs{{\boldsymbol{\vs}}}

\def\oq{{\overline{q}}}

\def\bari{\bar{i}}
\def\barj{\bar{j}}

\def\bA{\bar{A}}
\def\bB{\bar{B}}

\def\bF{\bar{F}}
\def\bR{\bar{R}}

\def\bW{\bar{W}}

\def\End{\text{End}}
\def\Res{\text{Res}}

\def\eEnd{\emph{End}}

\def\tCA{\widetilde{\CA}}

\def\zzz{\frac {\BZ^2}{(n,n)\BZ}}

\def\bF{\bar{F}}

\def\bQ{\bar{Q}}

\def\barc{\bar{c}}

\def\hdeg{\text{hdeg }}
\def\vdeg{\text{vdeg }}

\def\oo{\overline}

\def\fff{\BQ(q,\oq^{\frac 1n})}

\begin{document}
	
\title[Deformed $W$--algebras in type $A$ for rectangular nilpotent]{\Large{\textbf{Deformed $W$--algebras in type $A$ for rectangular nilpotent}}}
	
\author[Andrei Negu\cb t]{Andrei Negu\cb t}
\address{MIT, Department of Mathematics, Cambridge, MA, USA}
\address{Simion Stoilow Institute of Mathematics, Bucharest, Romania}
\email{andrei.negut@gmail.com}
	
\maketitle
	
\begin{abstract} For any $n,r \in \BN$, we construct an algebra $\CP_n^r$ via generators and quadratic relations, and show that it deforms the $W$--algebra of $\fgl_{nr}$ with respect to a nilpotent with Jordan block decomposition $r+...+r$. We introduce a surjective morphism from half of the quantum toroidal algebra of $\fgl_n$ to $\CP_n^r$, and show that the action of the quantum toroidal algebra on the $K$--theory groups of the moduli spaces of parabolic sheaves factors through $\CP_n^r$.
		
\end{abstract}

\section{Introduction}

\subsection{} Fix natural numbers $n, r$. We recall the quantum toroidal algebra $\UU$:
\begin{equation} 
\label{eqn:triangular 1}
\UU \cong \CS^+ \otimes \uui^{\otimes n} \otimes \CS^-
\end{equation}
where $\CS^\pm$ are the usual trigonometric type $\widehat{A}_n$ shuffle algebras (\cite{E, FO}) that we will recall in Section \ref{sec:shuf 2}. We will also need another triangular decomposition (\cite{PBW, Tale}):
\begin{equation} 
\label{eqn:triangular 2}
\UU \cong \CA^+ \otimes \uu \otimes \CA^-
\end{equation} 
where $\CA^\pm$ are a new kind of shuffle algebra, that we will recall in Section \ref{sec:shuf 1}. In \cite{RS}, the authors construct a surjective homomorphism from the Yangian of $\fsl_n$ to type $A$ finite $W$--algebras for rectangular nilpotent (generalized by \cite{BK} to arbitrary nilpotent by constructing certain shifted Yangians). In the present paper, we seek the $q$--deformed affinization of this principle. Explicitly, we define certain elements: 
\begin{equation}
\label{eqn:w gens}
W_{ij}^{(k)} \in \wCA^+, \quad \forall (i,j) \in \zzz, \ k \in \BN 
\end{equation}
(the completion is defined in \eqref{eqn:completion}) and show that they satisfy the quadratic relations \eqref{eqn:w rels}. If we let $\CP_n^\infty$ be the algebra generated by abstract symbols \eqref{eqn:w gens} modulo the corresponding quadratic relations, then our main result is: \\

\begin{theorem}
\label{thm:main 1}

There exists an isomorphism $\CP_n^\infty \cong \wCA^+$. \\ 	
	
\end{theorem} 

\noindent We then define the quotient:
\begin{equation}
\label{eqn:quotient}
\CP_n^r = \CP_n^\infty \Big / \left( W_{ij}^{(k)} \right)_{(i,j) \in \zzz}^{k > r}
\end{equation}
and identify $\CP_n^r$ with that quotient of (the completed half of) the quantum toroidal algebra which was shown in \cite{Par} to act on the $K$--theory groups of moduli spaces of sheaves on the plane with parabolic structure of type $(r,...,r)$. Thus, we obtain: \\

\begin{theorem} 
\label{thm:main 2}

There exists an action: 
\begin{equation}
\label{eqn:k-theory}
\CP_n^r \curvearrowright \bigoplus_{d_1,...,d_n = 0}^\infty K(\CM^{(r,...,r)}_{(d_1,...,d_n)})
\end{equation}
(we refer to \loccit for a review of the moduli space in the right-hand side). \\

\end{theorem} 

\noindent Combining Theorems \ref{thm:main 1} and \ref{thm:main 2} provides the deformed affine version of both the main result of \cite{RS} and of the AGT conjecture for sheaves with parabolic structure of type $(r,...,r)$, due to the following fact about the algebra $\CP_n^r$: \\

\begin{theorem}
\label{thm:main 3} 

The algebra $\CP_n^r$ deforms the universal enveloping algebra of the $W$--algebra of $\fgl_{nr}$ associated with nilpotent of Jordan type $r+...+r$. \\

\end{theorem}

\noindent Deformed $W$--algebras for principal nilpotent, i.e. corresponding to $r=1$ in our notation, have been studied in \cite{AKOS}, \cite{FF}, \cite{FR}. In fact, the results in the present paper can be interpreted as a generalization of the generators-and-relations presentation of \cite{AKOS} from the case of principal nilpotent to that of rectangular nilpotent. \\

\noindent In the case of nilpotent of any Jordan type $r_1+...+r_n$, we expect the appropriate analogue of $\CP_n^r$ to be a certain quotient of the \underline{shifted} quantum toroidal algebra of $\fgl_{n}$. We give a candidate for such a quotient in formula (1.7) of \cite{Par}, but the only evidence we have for this choice is the fact that it acts on the $K$--theory of moduli spaces of sheaves on the affine plane with parabolic structure of type $(r_1,...,r_n$). \\

\noindent In the non-deformed case, we note the parallel development of \cite{U}, where the author constructs a surjective homomorphism from the affine Yangian of $\fsl_n$ to $W$--algebras of type $A$ for rectangular nilpotent. While we expect the construction of \loccit to match ours, bridging the two languages is far from trivial. \\

\noindent I would like to thank Tomoyuki Arakawa, Mikhail Bershtein, Pavel Etingof, Boris Feigin, Roman Gonin, Victor Kac, Ryosuke Kodera, 
Tomas Proch\'azka, Junichi Shiraishi and Alexander Tsymbaliuk for valuable conversations about $W$--algebras. I gratefully acknowledge NSF grants DMS--1760264 and DMS--1845034, as well as support from the Alfred P. Sloan Foundation. \\

\section{Deformed $W$--algebras}
\label{sec:w}

\subsection{} Let $V$ be an $n$--dimensional vector space with a fixed basis. We will write:
$$
E_{ij} \in \text{Mat}_{n \times n} \cong \End(V)
$$ 
$\forall i,j \in \{1,...,n\}$ for the elementary matrix with a single $1$ at the intersection of row $i$ and column $j$. For a variable $x$, we will generalize the notation above to:
\begin{equation}
\label{eqn:notation 0}
E_{ij} = E_{\bari \barj} x^{\left \lfloor \frac {i-1}n \right \rfloor - \left \lfloor \frac {j-1}n \right \rfloor} \in \End(V)[x^{\pm 1}]
\end{equation}
for all $i,j \in \BZ$, where:
$$
\bari = i - n \left \lfloor \frac {i-1}n \right \rfloor 
$$
denotes the residue of any integer $i$ in the set $\{1,...,n\}$. Consider the $R$--matrix:
\begin{equation}
\label{eqn:def r}
R (x) \in \End_{\BQ(q)} (V \otimes V)
\end{equation}
given by:
\begin{equation}
\label{eqn:r}
R(x) = \sum_{1\leq i,j \leq n} E_{ii} \otimes E_{jj} \left(\frac {q - xq^{-1}}{1-x} \right)^{\delta_i^j} + (q - q^{-1}) \sum_{1 \leq i \neq j \leq n} E_{ij} \otimes E_{ji} \frac {x^{\delta_{i<j}}}{1-x}
\end{equation}
It is well-known that $R(x)$ satisfies the Yang-Baxter equation with parameter:
\begin{equation}
\label{eqn:ybe}
R_{12} \left( \frac {z_1}{z_2} \right) R_{13} \left( \frac {z_1}{z_3} \right) R_{23} \left( \frac {z_2}{z_3} \right) = R_{23} \left( \frac {z_2}{z_3} \right)  R_{13} \left( \frac {z_1}{z_3} \right) R_{12} \left( \frac {z_1}{z_2} \right) 
\end{equation}
as an equality in $\End(V^{\otimes 3})(z_1,z_2,z_3)$ (where $R_{12} = R \otimes 1$ etc). We note that:
\begin{equation}
\label{eqn:unitary}
R_{12} \left( \frac {z_1}{z_2} \right) R_{21} \left( \frac {z_2}{z_1} \right) = f \left( \frac {z_1}{z_2} \right) \cdot \text{Id}_{V \otimes V}
\end{equation}
where $R_{12} = R$, $R_{21} = R$ with the two tensor factor switched, and:
$$
f(x) = \frac {(1-xq^2)(1-xq^{-2})}{(1-x)^2}
$$

\subsection{} 
\label{sub:symbols} 

Let us consider symbols $W^{(k)}_{ij}$ for all $k \in \BN$ and $(i,j) \in \zzz$, and define:
\begin{equation}
\label{eqn:def p}
\CP_n^\infty = \bigoplus_{k,d \in \BZ} \mathop{\prod^{i_1-j_1+...+i_t-j_t = d}_{\frac {i_1-j_1}{k_1} \geq ... \geq \frac {i_t-j_t}{k_t}}}^{k_1+...+k_t = k} \BQ(p,q) \cdot W_{i_1j_1}^{(k_1)} ... W_{i_tj_t}^{(k_t)}
\end{equation}
as a vector space. In other words, an element of $\CP_n^\infty$ is an infinite sum of symbols $W_{i_1j_1}^{(k_1)} ... W_{i_tj_t}^{(k_t)}$ with $\frac {i_a-j_a}{k_a}$ in non-increasing order, but with $k_1+...+k_t$ and $i_1-j_1+...+i_t-j_t$ taking finitely many values. Consider the generating series:
\begin{equation}
\label{eqn:series}
W^{(k)}(x) = \sum_{(i,j) \in \zzz} W^{(k)}_{ij} \cdot E_{ij}
\end{equation}
which takes values in $\CP_n^\infty \otimes \End(V) [[x^{\pm 1}]]$. Let us write $(12) \in \text{End}(V \otimes V)$ for the permutation, and $\delta(z) = \sum_{k \in \BZ} z^k$ for the delta function doubly-infinite series. \\

\begin{definition} 
\label{def:w}
	
We make the $\BQ(p,q)$ vector space \eqref{eqn:def p} into an associative algebra with unit by imposing the following relations for all natural numbers $k \leq k'$:
$$
R_{12} \left(\frac xy \cdot p^{2(k-k')} \right) W^{(k)}_1(x) R_{21} \left(\frac yx \cdot p^{2k'}\right) W^{(k')}_2(y) \prod_{i=k'-k+1}^{k'-1} f \left(\frac xy \cdot p^{-2i} \right) - 
$$
$$
- W^{(k')}_2(y) R_{12} \left(\frac xy \cdot p^{2k}\right) W^{(k)}_1(x) R_{21} \left(\frac yx \right) \prod_{i=1}^{k-1} f \left( \frac yx \cdot p^{-2i} \right) = 
$$
$$
= \sum_{a \in \{-k',...,k-k'-1\} \sqcup \{1,...,k\}} \emph{sgn}(a) \cdot \delta \left(\frac {x}{y} \cdot p^{2a} \right) 
$$
\begin{equation}
\label{eqn:w rels}
(q^{-1}-q) (12) \cdot \left[ W^{(k'+a)}_1(x) R_{21} \left(p^{2k}\right) W^{(k-a)}_2(y) \prod_{i=1}^{k-1} f \left( p^{-2i} \right) \right]
\end{equation}
as an equality in $\CP_n^\infty \otimes \emph{End}_{\BQ(p,q)}(V \otimes V)[[x^{\pm 1}, y^{\pm 1}]]$. \footnote{In formula \eqref{eqn:w rels}, we write $W^{(k)}_1(x) = W^{(k)}(x) \otimes 1$ and $W^{(k')}_2(y) = 1 \otimes W^{(k')}(y)$.} \\

\end{definition}

\begin{remark}

The left-hand side of \eqref{eqn:w rels} is a particular case of the reflection equation (\cite{S}). In fact, the $r=1$ case of the construction of Subsection \ref{sub:miura} (which already featured in \cite{RS2}) implies that there exists an algebra homomorphism:
$$
\CP_n^\infty \rightarrow \wuu
$$
As $\CP_n^\infty \cong \wCA^+$ is half of $\UU$, this extends to an ``evaluation homomorphism":
\begin{equation}
\label{eqn:eval}
\UU \rightarrow \wuu 
\end{equation}
which splits the embedding $\uu \hookrightarrow \UU$ of the middle factor of \eqref{eqn:triangular 2}. The morphism \eqref{eqn:eval} is equivalent to the $\br = (1,...,1)$ case of the action (1.5) of \cite{Par}. \\

\end{remark}

\noindent Formula \eqref{eqn:w rels} is not clear as stated, so we will rephrase it in the following steps: \\

\begin{enumerate} 
	
\item in Subsection \ref{sub:first} we explain what it ``does" \\

\item in Subsection \ref{sub:second} we explain how it is defined \\

\item in Subsection \ref{sub:third} we explain how it makes $\CP_n^\infty$ into an algebra \\

\end{enumerate} 

\begin{definition} 
\label{def:wr}

For any $r \in \BN$, define:
$$
\CP_n^r = \CP_n^\infty \Big / (W_{ij}^{(k)})_{(i,j) \in \zzz}^{k > r}
$$
Clearly, if $k' > r$, both sides of \eqref{eqn:w rels} lie in the two-sided ideal $(W_{ij}^{(k)})_{(i,j) \in \zzz}^{k > r}$. \\

\end{definition}

\subsection{} 
\label{sub:first}

In Section \ref{sec:shuf 1}, we will give a representation of the generating series \eqref{eqn:series} in terms of a so-called shuffle algebra. With respect to a certain basis $\{v_i\}$ of this shuffle algebra, we will show that there exist matrix-valued rational functions:
\begin{equation}
\label{eqn:g}
G_i^{(k,k')}(x,y) \in \End(V \otimes V)(x,y) 
\end{equation}
such that:
$$
R_{12} \left(\frac {xp^{2k}}{yp^{2k'}} \right) W^{(k)}_1(x) R_{21} \left(\frac {yp^{2k'}}x  \right) W^{(k')}_2(y) \prod_{i=k'-k+1}^{k'-1} f \left(\frac x{yp^{2i}} \right) = \sum_i G^{(k,k')}_i(x,y) v_i
$$
expanded for $|x| \gg |y|$, and:
$$
W^{(k')}_2(y) R_{12} \left(\frac {xp^{2k}}y \right) W^{(k)}_1(x) R_{21} \left(\frac yx \right) \prod_{i=1}^{k-1} f \left( \frac y{xp^{2i}} \right) = \sum_i G^{(k,k')}_i(x,y) v_i
$$
expanded for $|x| \ll |y|$. Therefore, the left-hand side of formula \eqref{eqn:w rels} consists of the difference between the two expansions of the rational functions \eqref{eqn:g} at $|x|\gg|y|$ and $|x|\ll|y|$. According to the well-known identity for rational functions:
$$
\Big( F(z) \text{ expanded for }|z| \ll 1 \Big) - \Big( F(z) \text{ expanded for }|z| \gg 1 \Big) =
$$
\begin{equation}
\label{eqn:identity}
= - \sum_{i=1}^d \delta \left(\frac z{\alpha_i} \right) \underset{z = \alpha_i}{\text{Res}} \frac {F(z)}z
\end{equation}
where $\alpha_1,...,\alpha_d$ are the poles of $F$ different from $0$ and $\infty$, the right-hand side of formula \eqref{eqn:w rels} simply claims that the only poles of the rational functions \eqref{eqn:g} are:
\begin{equation}
\label{eqn:poles}
y = x p^{2a} \quad \text{for} \quad a \in \{-k',...,k-k'-1\} \sqcup \{1,...,k\}
\end{equation}
and the residue at such a pole is given by:
$$ 
\underset{y=xp^{2a}}{\text{Res}} \frac {G_i^{(k,k')}(x,y)}y = (12) \cdot
$$
$$
\begin{cases} (q-q^{-1})  G_{i,21}^{(k-a,k'+a)}(xp^{2a},x) R(p^{-2a})^{-1} \prod_{i=1}^{a} f(p^{-2i}) &\text{if } a > 0 \\
(q^{-1}-q)  R \left(p^{2(k'+a-k)} \right)^{-1} G^{(k'+a,k-a)}_i(x,xp^{2a}) \prod_{i=1}^{k-k'-a} f(p^{-2i}) &\text{if } a < 0 \end{cases}
$$

\subsection{}
\label{sub:second}

The main take-away from the heuristic discussion in the previous Subsection is that the two terms in the left-hand side of \eqref{eqn:w rels}, namely:
\begin{align*}
&A^{(k,k')}(x,y) = R_{12} \left(\frac {xp^{2k}}{yp^{2k'}} \right) W^{(k)}_1(x) R_{21} \left(\frac {yp^{2k'}}x \right) W^{(k')}_2(y) \prod_{i=k'-k+1}^{k'-1} f \left(\frac x{yp^{2i}} \right) \\
&B^{(k,k')}(x,y) = W^{(k')}_2(y) R_{12} \left(\frac {xp^{2k}}y \right) W^{(k)}_1(x) R_{21} \left(\frac yx \right) \prod_{i=1}^{k-1} f \left( \frac y{xp^{2i}} \right)
\end{align*}
should be thought of as the expansions (as $|x|\gg|y|$ and $|x| \ll|y|$, respectively) of one and the same matrix-valued rational function. To convert this into a precise abstract presentation of formulas \eqref{eqn:w rels}, we will use the language of \cite{AKOS} (although our presentation will be closer to the style of \cite{O}). For any scalar $\alpha$, we define:
\begin{equation}
\label{eqn:gamma}
\Gamma^{(k,k')}(x,x\alpha) = 
\end{equation}
$$
= \left( \frac {A^{(k,k')}(x,y)}{x\alpha - y} \text{ expanded for } |x| \gg |y| \right) - \left( \frac {B^{(k,k')}(x,y)}{x\alpha - y} \text{ expanded for } |x| \ll |y| \right)
$$
With this definition, \eqref{eqn:gamma} is an element of $\End(V \otimes V)[[x^{\pm 1}, y^{\pm 1}]]$ whose coefficients are infinite sums of products of the $W_{ij}^{(k)}$'s, but well-defined as an element of \eqref{eqn:def p}. If all power series involved were the expansion of a rational function $G$ as in the previous Subsection, then \eqref{eqn:gamma} would equal $G(x,x\alpha)$ plus contributions from the residues of $G$ at the poles $\notin \{0, \infty\}$. Then for any $k \leq k'$, we postulate that:
$$
A^{(k,k')}(x,x\alpha)  \quad \text{ and } \quad B^{(k,k')}(x,x\alpha)
$$
are given by the expression:
$$
\Gamma^{(k,k')}(x,x\alpha) + (q^{-1}-q) (12) \sum_{a \in \{1,...,k\}}  \frac {A_{21}^{(k-a,k'+a)}(xp^{2a},x) R(p^{-2a})^{-1} \prod_{i=1}^{a} f(p^{-2i})}{x(\alpha - p^{2a})} +
$$
$$
+ (q-q^{-1}) (12) \sum_{a \in \{-k',...,k-k'-1\}} \frac {R \left(p^{2(k'+a-k)} \right)^{-1} B^{(k'+a,k-a)}_{12}(x,xp^{2a}) \prod_{i=1}^{k-k'-a} f(p^{-2i})}{x(\alpha - p^{2a})}
$$
for any $\alpha \neq p^{2a}$ with $a$ as in \eqref{eqn:poles}. For any such $\alpha$, this allows one to successively write every $A^{(k,k')}(x,x\alpha)$ and $B^{(k,k')}(x,x\alpha)$ as a product of various $\Gamma^{(k,k')}(x,x\alpha)$, the matrix-valued rational functions $R$, and the scalar-valued rational functions $f$. If one replaces the right-hand side of \eqref{eqn:w rels} by the resulting expressions in $\Gamma$, $R$ and $f$, then formulas \eqref{eqn:w rels} are well-defined expressions in terms of the symbols \eqref{eqn:series}. \\

\subsection{} 
\label{sub:third}

In the preceding Subsection, we gave a precise recipe for how to make sense of \eqref{eqn:w rels} as an element of $\End(V \otimes V)[[x^{\pm 1}, y^{\pm 1}]]$ whose coefficients are infinite sums of products of the $W_{ij}^{(k)}$'s. Explicitly, we can take the coefficient of:
$$
E_{ij} \otimes E_{i'j'} = E_{\bari \barj} x^{\left \lfloor \frac {i-1}n \right \rfloor - \left \lfloor \frac {j-1}n \right \rfloor} \otimes E_{\bari' \barj'} y^{\left \lfloor \frac {i'-1}n \right \rfloor - \left \lfloor \frac {j'-1}n \right \rfloor}
$$
for all $(i,j), (i',j') \in \zzz$, and we obtain the relation:
$$
W_{ij}^{(k)} W_{i'j'}^{(k')} q^{\delta_{\bari'}^{\barj} - \delta_{\bari'}^{\bari}} - W_{i'j'}^{(k')}  W_{ij}^{(k)} q^{\delta_{\barj'}^{\barj} - \delta_{\barj'}^{\bari}} + 
$$
$$
+ \sum_{s-t+s'-t' = i - j + i' - j'}^{i-j < s-t, \ i'-j' > s'-t'} \gamma \cdot W_{st}^{(k)} W_{s't'}^{(k')} + \sum_{s-t+s'-t'  = i - j + i' - j'}^{i-j>s-t, \ i'-j' < s'-t'} \gamma \cdot W_{s't'}^{(k')} W_{st}^{(k)} =
$$ 
$$
= \sum_{a=1}^k \left[ \sum_{s-t+s'-t'  = i - j + i' - j'}^{s-t\text{ bounded below}} \gamma W_{st}^{(k-a)} W_{s't'}^{(k'+a)} + \sum_{s-t+s'-t'  = i - j + i' - j'}^{s-t\text{ bounded above}} \gamma W_{s't'}^{(k'+a)} W_{st}^{(k-a)} \right]
$$
for various coefficients $\gamma \in \BQ(p,q)$ that can be computed recursively. So if $\frac {i-j}{k} < \frac {i'-j'}{k'}$ regardless of whether $k \leq k'$ or $k \geq k'$, iterating the formula above yields:
$$
W_{ij}^{(k)} W_{i'j'}^{(k')} = \sum_{s-t+s'-t' = i - j + i' - j'}^{\frac {s-t}{k} \geq \frac {s'-t'}{k'}} \gamma \cdot W_{st}^{(k)} W_{s't'}^{(k')} + \sum_{s-t+s'-t' = i - j + i' - j'}^{\frac {s-t}{k} \leq \frac {s'-t'}{k'}} \gamma \cdot W_{s't'}^{(k')} W_{st}^{(k)} +
$$
$$
\sum_{a=1}^{\min(k,k')}  \left[ \sum_{s-t+s'-t'  = i - j + i' - j'}^{s-t\text{ bounded below}} \gamma W_{st}^{(k-a)} W_{s't'}^{(k'+a)} + \sum_{s-t+s'-t'  = i - j + i' - j'}^{s-t\text{ bounded above}} \gamma W_{s't'}^{(k'+a)} W_{st}^{(k-a)} \right]
$$
As in the proof of Theorem 3.13 of \cite{W surf}, the formula above can be iterated to write:
$$
W_{ij}^{(k)} W_{i'j'}^{(k')} = \sum_{l+l' = k+k'}^{\min(l,l') \leq \min(k,k')} \sum_{s-t+s'-t' = i - j + i' - j'}^{\frac {s-t}{l} \geq \frac {s'-t'}{l'}} \gamma \cdot W_{st}^{(l)} W_{s't'}^{(l')}
$$
for various coefficients $\gamma \in \BQ(p,q)$. Applying this formula repeatedly shows that the product of any two infinite sums of the form \eqref{eqn:def p} can be naturally transformed into an infinite sum of the same form, thus making $\CP_n^\infty$ into an associative algebra. \\

\begin{remark}

A more precise way to rephrase the latter statement is to define $\CP_n^\infty$ by taking the free associative algebra obtained by concatenating infinite sums of the form \eqref{eqn:def p}, and quotienting it by the two-sided ideal generated by relations \eqref{eqn:w rels}. Then the discussion in the present Subsection shows that any such concatenation (a priori an infinite sum which is not of the form \eqref{eqn:def p}) can be transformed using relations \eqref{eqn:w rels} into an infinite sum which is of the form \eqref{eqn:def p}. Thus the aforementioned quotient construction respects our notion of infinite sums. \\

\end{remark}

\subsection{} 
\label{sub:quantum group}

Recall the $R$--matrix of \eqref{eqn:def r}, and let us define the quantum affine group. \\

\begin{definition}
\label{eqn:quantum group}
	
Consider the algebra:
\begin{equation}
\label{eqn:two realizations 2}
\uu := \BQ(q) \Big \langle s_{[i;j)}, t_{[i;j)}, c \Big \rangle^{1 \leq i \leq n}_{i \leq j \in \BZ} \Big/ \text{relations \eqref{eqn:rtt 0}--\eqref{eqn:rtt 3}}
\end{equation}
where:
\begin{equation}
\label{eqn:rtt 0}
c \text{ is central, and } s_{[i;i)} t_{[i;i)} = 1
\end{equation}
\begin{align}
R\left( \frac xy \right) S_1(x) S_2(y) &=  S_2(y)   S_1(x) R\left( \frac xy \right) \label{eqn:rtt 1} \\
R\left( \frac xy \right) T_1(x) T_2(y) &=  T_2(y)   T_1(x) R\left( \frac xy \right) \label{eqn:rtt 2} \\
R\left( \frac {xc}{y} \right) S_1(x) T_2(y) &= T_2(y) S_1(x) R\left( \frac {x}{yc} \right) \label{eqn:rtt 3} 
\end{align}
where:
\begin{align}
&S(x) = \sum_{1 \leq i \leq n}^{i\leq j \in \BZ} s_{[i;j)} \cdot E_{\bari \barj} x^{\left \lfloor \frac {i-1}n \right \rfloor - \left \lfloor \frac {j-1}n \right \rfloor} \quad \in \quad \uu \otimes \eEnd(V)[[x^{-1}]] \label{eqn:series s} \\
&T(x) = \sum_{1 \leq i \leq n}^{i\leq j \in \BZ} t_{[i;j)} \cdot E_{\barj \bari} x^{\left \lfloor \frac {j-1}n \right \rfloor - \left \lfloor \frac {i-1}n \right \rfloor} \quad \in \quad \uu \otimes \eEnd(V)[[x]] \label{eqn:series t}
\end{align}
	
\end{definition}

\noindent There is a bialgebra structure on $\uu$, given by the coproduct:
\begin{align}
&\Delta(S(x)) = (1 \otimes S(x c_1)) \cdot (S(x) \otimes 1) \label{eqn:cop rtt 1 initial} \\ 
&\Delta(T(x)) = (1 \otimes T(x)) \cdot (T(xc_2) \otimes 1) \label{eqn:cop rtt 2 initial} 
\end{align}
(where $c_1 = c \otimes 1$ and $c_2 = 1 \otimes c$). This coproduct preserves the subalgebras:
\begin{equation}
\label{eqn:subalgebras}
\uug = \BQ (q) \langle s_{[i;j)} \rangle_{i\leq j} \qquad \text{and} \qquad \uul = \BQ (q) \langle t_{[i;j)} \rangle_{i\leq j}
\end{equation}
of $\uu$. Moreover, all aforementioned algebras are $\zz$--graded, via:
$$
\deg s_{[i;j)} = - \deg t_{[i;j)} = [i;j)
$$
where we write $\bs^i = (0,...,\underbrace{1}_{\text{position }\bari},...,0)$ for any integer $i$, and:
$$
[i;j) = \begin{cases} \bs^i+...+\bs^{j-1} &\text{if } i \leq j \\-\bs^j - ... - \bs^{i-1} &\text{if } i>j \end{cases}
$$
The algebra $\uu$ has a basis given by the generators $s_{[i;j)}$ and $t_{[i;j)}$ taken in non-increasing order of their total degree, which is defined as $|\deg| \in \BZ$. We will write: 
\begin{equation}
\label{eqn:psi}
\psi_k = s_{[k;k)}^{-1} = t_{[k;k)}
\end{equation}
for all $1 \leq k \leq n$, and extend this notation to all integers $k$ by setting $\psi_{k+n} = c \psi_k$. \\

\subsection{}
\label{sub:new}

We now consider the $c = p^{-1}$ specialization of the quantum group, and replace the generating series $S(x)$, $T(x)$ satisfying \eqref{eqn:rtt 1}--\eqref{eqn:rtt 3} by generating series:
\begin{align}
&\Lambda^+(x) = \sum_{1 \leq i \leq n}^{i\geq j \in \BZ} v^+_{ij} \cdot E_{\bari \barj} x^{\left \lfloor \frac {i-1}n \right \rfloor - \left \lfloor \frac {j-1}n \right \rfloor} \quad \in \quad \frac {\uug}{c-p^{-1}} \otimes \End(V) [[x]] \label{eqn:a plus} \\
&\Lambda^-(x) = \sum_{1 \leq i \leq n}^{i\leq j \in \BZ} v^-_{ij} \cdot E_{\bari \barj} x^{\left \lfloor \frac {i-1}n \right \rfloor - \left \lfloor \frac {j-1}n \right \rfloor} \quad \in \quad \frac {\uul}{c-p^{-1}} \otimes \End(V) [[x^{-1}]] \label{eqn:a minus}
\end{align}
which satisfy the following relations:
\begin{align}
R\left( \frac xy \right) \Lambda^+_1(x) \Lambda^+_2(y) &=  \Lambda^+_2(y) \Lambda^+_1(x) R\left( \frac xy \right) \label{eqn:rtt 1 new} \\
\Lambda^-_1(x) \Lambda^-_2(y) R\left( \frac xy \right) &= R\left( \frac xy \right) \Lambda^-_2(y) \Lambda^-_1(x)  \label{eqn:rtt 2 new} 
\end{align} 
and:
\begin{equation}
\label{eqn:rtt 3 new} 
\Lambda_1^-(x) R_{21}\left( \frac {yp^2}{x} \right) \Lambda_2^+(y) = \Lambda_2^+(y) R_{12} \left( \frac xy \right)^{-1} \Lambda_1^-(x) 
\end{equation}
\footnote{To make relations \eqref{eqn:rtt 1}--\eqref{eqn:rtt 3} give rise to \eqref{eqn:rtt 1 new}--\eqref{eqn:rtt 3 new}, we set $\Lambda^+(x) = S\left( \frac {p^2}x \right)^\dagger$, while: 
$$
\Lambda^-(x) = D \cdot T \left( \frac {p^2}{xq^{2n}} \right)^{-1,\dagger} \cdot D^{-1} \cdot \exp \left( \sum_{k=1}^\infty \frac {p_{-k}x^k}k \cdot \text{Id} \right)
$$ 
where $p_{-1},p_{-2},...$ are primitve elements of $\uum$ (see \cite{Tor} for a review), while $\dagger$ denotes matrix transpose and $D = \text{diag}(q^2,...,q^{2n})$. This is not the only choice, but is the one featuring in \cite{Par}.} We consider the following completion of the algebra $\frac {\uu}{c-p^{-1}}$:
$$
U^{(1)} = \bigoplus_{d \in \BZ} \left ( \mathop{\prod^{\sum_{s=1}^{k+l} (i_s-j_s) = d}_{i_1-j_1 \geq ... \geq i_k-j_k \geq 0}}_{0 \geq i_{k+1}-j_{k+1} \geq ... \geq i_{k+l}-j_{k+l}} \BQ(p,q) \cdot v^+_{i_1j_1}... v^+_{i_kj_k} v^-_{i_{k+1}j_{k+1}}... v^-_{i_{k+l}j_{k+l}} \right )
$$
The following is a well-defined element of $U^{(1)} \otimes \End(V) [[x^{\pm 1}]]$:
\begin{equation}
\label{eqn:a}
\Lambda(x) = \Lambda^+(x) \Lambda^-(x) =: \sum_{1 \leq i \leq n}^{j \in \BZ} v_{ij} \cdot E_{\bari \barj} x^{\left \lfloor \frac {i-1}n \right \rfloor - \left \lfloor \frac {j-1}n \right \rfloor}
\end{equation}
If we call the number $i-j$ the total degree of the elements $v_{ij}^+$, $v_{ij}^-$, $v_{ij}$, then a product of such elements taken in non-increasing order of the total degree will be called normally ordered. It is a straightforward consequence of \eqref{eqn:rtt 1 new}--\eqref{eqn:rtt 3 new} that any product of normally ordered expressions can be normally ordered. Moreover:
$$
v_{i_1j_1}... v_{i_{k+l}j_{k+l}} = v^+_{i_1j_1}... v^+_{i_kj_k} v^-_{i_{k+1}j_{k+1}}... v^-_{i_{k+l}j_{k+l}} + \text{``smaller terms"}
$$ 
for any sequence $\sigma =  (i_1-j_1 \geq ... \geq i_k-j_k \geq 0 \geq i_{k+1}-j_{k+1} \geq ... \geq i_{k+l}-j_{k+l})$, where ``smaller terms" refers to normally ordered products whose sequence of total degrees is lexicographically greater than $\sigma$ (see the Appendix of \cite{Tor} for details). \\

\subsection{} 

For any $r \in \BN$, consider $r$ copies of the symbols defined above, denoted by:
\begin{equation}
\label{eqn:symbols}
v^{(a)}_{ij} \text{ and } v^{(a)\pm}_{ij} \qquad \forall a \in \{1,...,r\}
\end{equation}
and we impose upon them the same algebra relations as follow from \eqref{eqn:rtt 1 new}--\eqref{eqn:rtt 3 new}. \\

\begin{definition}
	
For symbols \eqref{eqn:symbols}, we consider the vector space:
$$
U^{(r)} = \bigoplus_{d \in \BZ} \left ( \prod^{\sum_{a=1}^r \sum_{s=1}^{k_a} (i^{(a)}_s-j^{(a)}_s) = d}_{i_1^{(a)}-j^{(a)}_1 \geq ... \geq i^{(a)}_{k_a}-j^{(a)}_{k_a}} \BQ(p,q) \cdot v^{(1)}_{i^{(1)}_1j^{(1)}_1}... v^{(1)}_{i^{(1)}_{k_1} j^{(1)}_{k_1}} ... v^{(r)}_{i^{(r)}_1j^{(r)}_1}... v^{(r)}_{i_{k_r}^{(r)} j^{(r)}_{k_r}}\right )
$$	
We make $U^{(r)}$ into an algebra by imposing relations \eqref{eqn:rtt 1 new}--\eqref{eqn:rtt 3 new} for the series $\Lambda^{(a)\pm}(x)$ and $\Lambda^{(a)}(x)$ for all $a \in \{1,...,r\}$ individually, as well as the relations:
\begin{equation}
\label{eqn:rtt final}
\Lambda^{(b)}_1(x) R_{21} \left(\frac {yp^2}x \right) \Lambda_2^{(a)}(y) = 
\end{equation}
$$
= R_{21} \left( \frac yx \right) \Lambda_2^{(a)}(y) R_{12} \left( \frac {xp^2}y \right) \Lambda_1^{(b)}(x) R_{12} \left( \frac xy \right)^{-1}
$$
in $U^{(r)} \in \eEnd(V \otimes V)[[x^{\pm 1}, y^{\pm 1}]]$ (expanded as $|x| \gg |y|$), for all $1 \leq a < b \leq r$. \\	
	
\end{definition}

\noindent A product of $v^{(a)}_{ij}$ in non-decreasing order of $a$, and in non-increasing order of $i-j$ to break ties, will be called normally ordered. When one needs to multiply two normally ordered products, one first uses \eqref{eqn:rtt final} to express:
$$
v^{(b)}_{st} v^{(a)}_{ij} = \sum^{i'-j' \leq i-j}_{s'-t' \geq s-t} \gamma \cdot v^{(a)}_{i'j'} v^{(b)}_{s't'}
$$
for all $i,j,s,t$ and certain coefficients $\gamma \in \BQ(p,q)$ that can be explicitly computed. Then, once all the $v^{(a)}_{ij}$'s are in non-decreasing order of $a$, one uses \eqref{eqn:rtt 1 new}--\eqref{eqn:rtt 3 new} to also place them in non-increasing order of $i-j$. This makes $U^{(r)}$ into an algebra, and we leave the details of this procedure as an exercise to the interested reader. \\

\subsection{}
\label{sub:miura}

Let $D$ be the difference operator $f(x) \leadsto f(xp^2)$. \\

\begin{definition}
\label{def:miura} 

Consider the series $\bW^{(1)}(x),...,\bW^{(r)}(x) \in U^{(r)} \otimes \eEnd(V) [[x^{\pm 1}]]$:
\begin{equation}
\label{eqn:miura}
\sum_{k=0}^{r} (-D)^{r-k} \cdot \bW^{(k)}(x) = \prod_{i=1}^r \Big(\Lambda^{(i)} (xp^{2(r-i)}) - D\Big)
\end{equation}
i.e.:
\begin{equation}
\label{eqn:bar w}
\bW^{(k)}(x) = \sum_{1 \leq i_1 < ... < i_k \leq r} \Lambda^{(i_1)}(xp^{2(k-1)}) ... \Lambda^{(i_{k-1})}(xp^2) \Lambda^{(i_k)}(x)
\end{equation}
and $\bW^{(k)}(x) = 0$ for $k > r$. \\

\end{definition}

\begin{proposition}
\label{prop:miura} 

The subalgebra of $U^{(r)}$ generated by the currents $\bW^{(k)}(x)$ deforms the universal enveloping algebra of the $W$--algebra of $\fgl_{nr}$ associated with nilpotent of Jordan type $r+...+r$. \\

\end{proposition}

\begin{proof} We will compare \eqref{eqn:miura} with the quantum Miura transformation for $W$--algebras in type $A$ associated with rectangular nilpotent (\cite{AM, EP}). To this end, let:
$$
q = e^{\e}, \quad p = e^{\e\beta} \quad \text{and} \quad \Lambda^{(i)}(x) = 1 + \e \bar{\Lambda}^{(i)}(x) + O(\e^2)
$$
and take the leading order term in $\e$ in all of our formulas. Since:
$$
R(z) = 1 + \e \bR(z) + O(\e^2), \quad \text{where} \quad \bR(z) = \sum_{1 \leq i , j \leq n} E_{ij} \otimes E_{ji} \cdot \frac {z^{\delta_{i \leq j}} + z^{\delta_{i<j}}}{1-z} 
$$
It is elementary to see that relations \eqref{eqn:rtt 1 new}--\eqref{eqn:rtt 3 new} mean that the coefficients of each series $\bar{\Lambda}^{(i)}(x)$ satisfy the commutation relations of $\fgl_n[t^{\pm 1}]$ (in fact, this is simply restating the well-known fact that the algebra \eqref{eqn:two realizations 2} deforms $\fgl_n[t^{\pm 1}]$). \\

\begin{claim}
\label{claim:commute}
	
The coefficients of series $\bar{\Lambda}^{(i)}(x)$ and $\bar{\Lambda}^{(j)}(x)$ commute for all $i \neq j$. \\

\end{claim} 

\noindent We will prove the claim after we complete the proof of the Proposition. As:
$$
D = e^{2\e \beta x \partial_x}
$$
the leading order term in $\e$ of the right-hand side of \eqref{eqn:miura} is precisely the quantum Miura transformation for $W$--algebras in type $A$ associated with rectangular nilpotent. The coefficients of the latter expression are the generators of the universal enveloping algebra of the $W$--algebra of $\fgl_{nr}$ associated with rectangular nilpotent, which concludes the proof of Proposition \ref{prop:miura}.

\begin{proof} \emph{of Claim \ref{claim:commute}:} Relation \eqref{eqn:rtt final} can be rewritten as:
$$
\Big[1 + \e \bar{\Lambda}^{(j)}_1(x) \Big] \left[1 + \e \bR_{21} \left(\frac {yp^2}x \right) \right] \Big[1 + \e \bar{\Lambda}_2^{(i)}(y) \Big] \left[1 + \e \bR_{12} \left( \frac xy \right) \right] = 
$$
$$
= \left[1 + \e \bR_{21} \left( \frac yx \right) \right] \Big[1 + \e \bar{\Lambda}_2^{(i)}(y) \Big] \left[1 + \e \bR_{12} \left( \frac {xp^2}y \right) \right] \Big[1 + \e \bar{\Lambda}^{(j)}_1(x) \Big]
$$
plus $O(\e^2)$. We note that the $O(\e^2)$ terms from each bracket do not contribute anything to the above relation, because their contributions to the left and the right-hand sides cancel. Therefore, we may equate the $\e^2$ terms in the left and right-hand sides of the above relation, and obtain:
$$
\bar{\Lambda}^{(j)}_1(x) \bar{\Lambda}_2^{(i)}(y) - \bar{\Lambda}_2^{(i)}(y) \bar{\Lambda}^{(j)}_1(x) + \squiggly{\bR_{21} \left(\frac {yp^2}x \right) \bR_{12} \left( \frac xy \right) - \bR_{21} \left( \frac yx \right)\bR_{12} \left( \frac {xp^2}y \right)} +
$$
$$
+ \squiggly{\bar{\Lambda}^{(j)}_1(x) \bR_{21} \left(\frac {yp^2}x \right) + \bar{\Lambda}^{(j)}_1(x) \bR_{12} \left( \frac xy \right) - \bR_{21} \left( \frac yx \right) \bar{\Lambda}^{(j)}_1(x) - \bR_{12} \left( \frac {xp^2}y \right) \bar{\Lambda}^{(j)}_1(x)} +
$$
$$
+ \squiggly{\bR_{21} \left(\frac {yp^2}x \right) \bar{\Lambda}_2^{(i)}(y) + \bar{\Lambda}_2^{(i)}(y)\bR_{12} \left( \frac xy \right) - \bR_{21} \left( \frac yx \right)\bar{\Lambda}_2^{(i)}(y) - \bar{\Lambda}_2^{(i)}(y)\bR_{12} \left( \frac {xp^2}y \right)} = 
$$
\begin{equation}
\label{eqn:four lines}
= \frac 1{\e} \left[ \bR_{12} \left( \frac {xp^2}y \right) - \bR_{12} \left( \frac {x}y \right) - \bR_{21} \left(\frac {yp^2}x \right) + \bR_{21} \left(\frac {y}x \right) \right]
\end{equation}
up to order $\e$. The terms with the squiggly underline on the first and third rows vanish (up to order $\e$) for trivial reasons, while the term with the squiggly underline on the second row vanishes (up to order $\e$) due to the elementary identity:
$$
\bR_{21} \left( \frac 1z \right) = - \bR_{12}(z)
$$
Meanwhile, the right-hand side of \eqref{eqn:four lines} vanishes due to the identity:
$$
\bR_{12}(zp^2) - \bR_{12}(z) = - 4 \e \beta (12) \cdot \frac z{(1-z)^2} + O(\e^2)
$$
so we are left with $[\bar{\Lambda}^{(j)}_1(x), \bar{\Lambda}_2^{(i)}(y)] = 0$, as we needed to prove. 
\end{proof} 

\end{proof} 

\begin{proof} \emph{of Theorem \ref{thm:main 3}:} We must prove that relations \eqref{eqn:w rels} hold $W^{(k)}(x) \leadsto \bW^{(k)}(x)$ and to this end, we will use the approach used by \cite{O} in the case $n=1$. To be more precise, we will prove the following: \\
	
\begin{claim}
\label{claim:precise 1}

For any $k \leq k'$, the expansions:
\begin{equation}
\label{eqn:un}
R_{12} \left(\frac {xp^{2k}}{yp^{2k'}} \right) \bW^{(k)}_1(x) R_{21} \left(\frac {yp^{2k'}}x \right) \bW^{(k')}_2(y) \prod_{i=k'-k+1}^{k'-1} f \left(\frac x{y p^{2i}} \right) 
\end{equation}
(as $|x| \gg |y|$) and:
\begin{equation}
\label{eqn:doi}
\bW^{(k')}_2(y) R_{12} \left(\frac {xp^{2k}}y \right) \bW^{(k)}_1(x) R_{21} \left(\frac yx \right) \prod_{i=1}^{k-1} f \left( \frac y{x p^{2i}} \right) 
\end{equation}
(as $|x| \ll |y|$) are both represented by the same element of $U^{(r)} \otimes \emph{End}(V \otimes V) (x,y)$. \\

\end{claim}

\begin{claim}
\label{claim:precise 2}

As rational functions of $x$ and $y$, the only poles of \eqref{eqn:un}--\eqref{eqn:doi} are:
\begin{equation}
\label{eqn:pole} 
y = xp^{2a} \text{ for } a \in \{1,...,k\} \sqcup \{-k',...,-k'+k-1\}
\end{equation}
and the corresponding residue is:
\begin{equation}
\label{eqn:residue}
\emph{sgn}(a) (q - q^{-1}) (12) \left[ \bW^{(k'+a)}_1(x) R_{21} \left(p^{2k}\right) \bW^{(k-a)}_2(y) \prod_{i=1}^{k-1} f \left( p^{-2i} \right) \right] \Big|_{y = xp^{2a}}
\end{equation}
Note that $y = xp^{2a}$ is not among the poles of the expression \eqref{eqn:un}--\eqref{eqn:doi} when $(k,k')$ are replaced by $(k'+a,k-a)$ with $a$ as above, so \eqref{eqn:residue} is well-defined. \\ 

\end{claim}

\noindent We will prove Claims \ref{claim:precise 1} and \ref{claim:precise 2} simultaneously, by induction on $k$, so let us first take care of the base case $k=1$. To this end, we need to show how to normally order products of $\Lambda^{(i)}(x)$'s. We will find it easier to work with the notation:
\begin{equation}
\label{eqn:not}
\Lambda^{(i_a,...,i_b)} = \Lambda^{(i_a)}(xp^{2(k-a)}) ... \Lambda^{(i_b)}(xp^{2(k-b)})
\end{equation}
for all $i_1<...< i_a < ... < i_b < ... < i_k \in \{1,...,r\}$ (the numbers $a,b,k$, although not part of the notation \eqref{eqn:not}, will always be clear from context). By iterating relation \eqref{eqn:rtt final}, we obtain the following formula if $j_1 < ... < j_k < i$:
$$
\Lambda_1^{(i)} R_{21} \left(\frac {yp^{2k}}x \right) \Lambda_2^{(j_1,...,j_k)} = R_{12} \left( \frac x{y p^{2(k-1)}} \right)^{-1} \Lambda_2^{(j_1,...,j_k)} R_{12} \left( \frac {xp^2}y \right) \Lambda_1^{(i)} R_{21} \left( \frac yx \right) 
$$
while if $i < j_1 < ... < j_k$:
$$
\Lambda_2^{(j_1,...,j_k)} R_{12} \left(\frac {xp^2}y \right) \Lambda_1^{(i)} = R_{12} \left( \frac x{y p^{2(k-1)}} \right) \Lambda_1^{(i)} R_{21} \left( \frac {yp^2}x \right) \Lambda_2^{(j_1,...,j_k)} R_{21} \left( \frac yx \right)^{-1}
$$	
Throughout this proof, all series denoted by $\Lambda_1$ will have variable $x$, and all series denoted by $\Lambda_2$ will have variable $y$. Using the identities above, we have the following normal-ordering formulas for any numbers $i$ and $j_1 < ... < j_{k'} \in \{1,...,r\}$:
\begin{equation}
\label{eqn:for 1}
R_{12} \left(\frac {xp^2}{yp^{2k'}} \right) \Lambda^{(i)}_1 R_{21} \left(\frac {yp^{2k'}}x \right) \Lambda^{(j_1,...,j_{k'})}_2 =
\end{equation}
$$
= \Lambda_2^{(1,...,j_{b-1})} R_{12} \left( \frac {x}{yp^{2(k'-b)}} \right) \Lambda_1^{(i)} R_{21} \left( \frac {yp^{2(k'-b+1)}}{x} \right) \Lambda_2^{(j_b,...,j_{k'})} 
$$
where $b$ is such that $j_{b-1} < i \leq j_b$, and:
\begin{equation}
\label{eqn:for 2}
\Lambda^{(j_1,...,j_{k'})}_2 R_{12} \left(\frac {xp^2}y \right) \Lambda^{(i)}_1 R_{21} \left(\frac yx \right) = 
\end{equation}
$$
= \Lambda_2^{(j_1,...,j_{b'-1})} R_{12} \left( \frac x{yp^{2(k'-b')}} \right) \Lambda_1^{(i)} R_{21} \left( \frac {yp^{2(k'-b'+1)}}x \right) \Lambda_2^{(j_{b'},...,j_{k'})} 
$$
where $b'$ is such that $j_{b'-1} \leq i < j_{b'}$. Here we have two cases: \\

\begin{itemize}[leftmargin=*]
	
\item if $i = j_c$ for some $c$, then formulas \eqref{eqn:rtt 1 new}--\eqref{eqn:rtt 3 new} imply that the right-hand sides of \eqref{eqn:for 1}--\eqref{eqn:for 2} are both equal to the following normally ordered expression:
\begin{equation}
\label{eqn:for 3}
\Lambda_2^{(j_1,...,j_{c-1})} \Lambda_2^{(c)+} \Lambda_1^{(i)+} \Lambda_1^{(i)-} \Lambda_2^{(c)-}  \Lambda_2^{(j_{c+1},...,j_{k'})} 
\end{equation}

\item if $i \notin \{j_1,...,j_{k'}\}$, then $b=b'$, hence the RHS of \eqref{eqn:for 1} and \eqref{eqn:for 2} are equal. \\

\end{itemize} 

\noindent The bullets above prove the $k=1$ case of Claim \ref{claim:precise 1}. As for Claim \ref{claim:precise 2}, we need to study the poles of \eqref{eqn:for 1}--\eqref{eqn:for 2} and \eqref{eqn:for 3} as rational functions in $x$ and $y$. The latter does not have any poles, while those of the former all come from the unique poles of the $R$--matrix:
\begin{equation}
\label{eqn:a1}
y = \frac x{p^{2(k'-b)}}, \ \ \text{with residue } \Lambda_2^{(j_1,...,j_{b-1})} (q-q^{-1})(12) \Lambda_1^{(i)} R_{21} (p^2) \Lambda_2^{(j_b,...,j_{k'})} 
\end{equation}
\begin{equation}
\label{eqn:a2}
y = \frac x{p^{2(k'-b+1)}}, \text{with residue } \Lambda_2^{(j_1,...,j_{b-1})} R_{12} (p^2) \Lambda_1^{(i)}  (q^{-1}-q)(12) \Lambda_2^{(j_b,...,j_{k'})} 
\end{equation}
We note that \eqref{eqn:a1} for a collection $(i,...<j_b < ... )$ precisely cancels out \eqref{eqn:a2} for the collection $(j_b, ...< i < ... )$. Therefore, as we sum over all $i$ and $j_1 < ... < j_{k'}$ in $\{1,...,r\}$, the only residues which survive are \eqref{eqn:a1} when $i > j_{k'}$ and \eqref{eqn:a2} when $i < j_1$. We conclude that the poles of the rational function \eqref{eqn:un}--\eqref{eqn:doi} are:
\begin{align*}
&y = xp^2, \quad \text{with residue } \sum_{j_1 < ... < j_{k'} < i} \delta \left(\frac y{xp^2} \right) \Lambda_2^{(j_1,...,j_{k'})} (q-q^{-1})(12) \Lambda_1^{(i)} R_{21}(p^2) \\
&y = \frac x{p^{2k'}}, \text{ with residue } \sum_{i < j_1 < ... < j_{k'}} \delta \left(\frac {yp^{2k'}}{x} \right) R_{12}(p^2) \Lambda_1^{(i)} (q^{-1}-q)(12) \Lambda_2^{(j_1,...,j_{k'})}
\end{align*}
This is precisely expression \eqref{eqn:residue} for $k=1$, which concludes the base case of Claims \ref{claim:precise 1} and \ref{claim:precise 2}. In particular, the fact that:
$$
\underset{z = xp^2}{\text{Res}} \left[ \bW^{(k-1)}_\bullet(z) R_{1\bullet} \left(\frac {xp^2}z \right) \bW^{(1)}_1(x) R_{\bullet 1} \left(\frac zx\right) \right] = (q-q^{-1})(1\bullet) \bW_1^{(k)}(x) R_{\bullet 1} (p^2)
$$
(we use the notation $z,\bullet$ instead of $y,2$) implies that:
\begin{equation}
\label{eqn:iterate}
\bW_1^{(k)}(x)  = \underset{z = xp^2}{\text{Res}} \left[ \frac {(1\bullet)}{q-q^{-1}} \cdot \bW^{(k-1)}_\bullet(z) R_{1\bullet} \left(\frac {xp^2}z \right) \bW^{(1)}_1(x) \right] 
\end{equation}
In other words, the higher currents can be expressed as iterated residues of products of lower currents. Let us now assume that Claims \ref{claim:precise 1} and \ref{claim:precise 2} are proved for $k-1$ and deduce them for $k$, thus proving the induction step. By \eqref{eqn:iterate}, we have:
$$
\Big(\text{formula \eqref{eqn:un}} \Big) = \underset{z = xp^2}{\text{Res}} \left[\frac {(1\bullet)}{q-q^{-1}} \cdot \right.
$$
$$
\left. R_{\bullet 2} \left(\frac {xp^{2k}}{yp^{2k'}} \right)  \bW^{(k-1)}_\bullet(z) R_{1\bullet} \left(\frac {xp^2}z \right) \bW^{(1)}_1(x) R_{21} \left(\frac {yp^{2k'}}x \right) \bW^{(k')}_2(y) \prod_{i=k'-k+1}^{k'-1} f \left(\frac x{y p^{2i}} \right) \right]
$$
$$
\Big(\text{formula \eqref{eqn:doi}} \Big) = \underset{z = xp^2}{\text{Res}} \left[\frac {(1\bullet)}{q-q^{-1}} \cdot \right.
$$
$$
\left. \bW^{(k')}_2(y) R_{\bullet 2} \left(\frac {xp^{2k}}y \right) \bW^{(k-1)}_\bullet(z) R_{1\bullet} \left(\frac {xp^2}z \right) \bW^{(1)}_1(x) R_{21} \left(\frac yx \right) \prod_{i=1}^{k-1} f \left( \frac y{x p^{2i}} \right)  \right]
$$
as identities in $\End(V \otimes V \otimes V)$, with the three copies of $V$ being denoted by $1,2,\bullet$. By the induction hypothesis of Claim \ref{claim:precise 1}, we can rewrite the formulas above as:
\begin{equation}
\label{eqn:mi}
\Big(\text{formula \eqref{eqn:un}} \Big) = \underset{z = xp^2}{\text{Res}} \left[\frac {(1\bullet)}{q-q^{-1}} R_{\bullet 2} \left(\frac {zp^{2k-2}}{yp^{2k'}} \right)  \bW^{(k-1)}_\bullet(z)  \right.
\end{equation} 
$$
\left. \squiggly{R_{1\bullet} \left(\frac {xp^2}z \right)} R_{21} \left(\frac {yp^{2k'}}z \right) \bW^{(k')}_2(y) R_{12} \left(\frac {xp^2}y \right) \bW^{(1)}_1(x) R_{21} \left(\frac yx \right)   \prod_{i=k'-k+1}^{k'-2} f \left(\frac x{y p^{2i}} \right) \right]
$$
and:
\begin{equation}
\label{eqn:pi}
\Big(\text{formula \eqref{eqn:doi}} \Big) = \underset{z = xp^2}{\text{Res}} \left[\frac {(1\bullet)}{q-q^{-1}} R_{\bullet 2} \left( \frac {zp^{2k-2}}{yp^{2k'}} \right) \bW^{(k-1)}_\bullet(z) \right.
\end{equation}
$$
\left. R_{2\bullet}  \left(\frac {yp^{2k'}}z \right)  \bW^{(k')}_2(y) R_{\bullet 2} \left( \frac {xp^2}y \right) \squiggly{R_{1\bullet} \left(\frac {xp^2}z \right)} \bW^{(1)}_1(x) R_{21} \left(\frac yx \right) \prod_{i=k'-k+1}^{k'-2} f \left( \frac x{y p^{2i}} \right) \right]  
$$
We observe that the right-hand sides of the formulas above are represented by the same rational function, because the factor with the squiggly underline has residue at $z = xp^2$ equal to a multiple of the permutation matrix $(1\bullet)$, hence it can be moved left-to-right by changing the indices of the various other factors along the way. \footnote{However, the factor $R_{1\bullet} \left(\frac {xp^2}z \right)$ cannot be moved past $\bW^{(k-1)}_\bullet(z)$ and $\bW_1^{(1)}(x)$ without changing the overall residue; fortunately, our argument does not require this.} This establishes the induction step of Claim \ref{claim:precise 1}. \\

\noindent As for the induction step of Claim \ref{claim:precise 2}, we need to identify all the poles and the corresponding residues of the rational functions \eqref{eqn:mi}--\eqref{eqn:pi}. In general, if we have rational functions $F(z,y)$ with poles $\{y=z\alpha\}$ and $G(y,x)$ with poles $\{y = x\beta\}$, then the poles of the rational function $F(z,y)G(y,x)$ in the variable $y$ are:
\begin{align*}
&y = z\alpha \quad \text{with residue} \quad \left[ \underset{y = x\alpha}{\text{Res}} F(z,y) \right] G(z\alpha,x) \\
&y = x\beta \quad \text{with residue} \quad F(z,x\beta )\left[ \underset{y = x\beta}{\text{Res}} G(y,x) \right] 
\end{align*}
as long as all the poles are simple (which will be the case in our situation). In the case at hand, the induction hypothesis implies that:
$$
\bW^{(k')}_2(y) R_{12} \left(\frac {xp^2}y \right) \bW^{(1)}_1(x) R_{21} \left(\frac yx \right) 
$$
has poles:
\begin{align}
&y = x p^2 \quad \text{with residue } \quad (q-q^{-1})(12) \bW_1^{(k'+1)}(x) R_{21}(p^2) \label{eqn:pole 1} \\
&y = \frac x{p^{2k'}} \quad \text{with residue} \quad (q^{-1}-q)(12) R_{21}(p^2) \bW_2^{(k'+1)} \left( \frac x{p^{2k'}} \right) \label{eqn:pole 2}
\end{align}
and that:
$$
R_{\bullet 2} \left( \frac {zp^{2k-2}}{yp^{2k'}} \right) \bW^{(k-1)}_\bullet(z) R_{2\bullet}  \left(\frac {yp^{2k'}}z \right)  \bW^{(k')}_2(y) \prod_{i=k'-k+2}^{k'-1} f \left(\frac z{y p^{2i}} \right) 
$$
has poles:
\begin{equation}
\label{eqn:pole 3} 
y = z p^{2b} \quad \text{with residue} \quad (q-q^{-1}) (\bullet 2) \cdot
\end{equation}
$$
\left[ \bW^{(k'+b)}_\bullet(z) R_{2\bullet} \left(p^{2k-2}\right) \bW^{(k-b-1)}_2(zp^{2b}) \prod_{i=1}^{k-2} f \left( p^{-2i} \right) \right] 
$$
\begin{equation}
\label{eqn:pole 4} 
y = z p^{2c} \quad \text{with residue} \quad (q^{-1}-q) (\bullet 2) \cdot 
\end{equation}
$$
\left[ \bW^{(k'+c)}_\bullet(z) R_{2\bullet} \left(p^{2k-2}\right) \bW^{(k-c-1)}_2(zp^{2c}) \prod_{i=1}^{k-2} f \left( p^{-2i} \right) \right] 
$$
where $b \in \{1,...,k-1\}$ and $c \in \{-k',...,-k'+k-2\}$. Since ultimately we will take the residue at $z = xp^2$, it is easy to see that the poles \eqref{eqn:pole 1}--\eqref{eqn:pole 4} are the same as those prescribed in \eqref{eqn:pole}. Let us now check that the residues also match up with \eqref{eqn:residue}. We will do the computation for the poles \eqref{eqn:pole 2} and \eqref{eqn:pole 4}, and leave the analogous cases of the poles \eqref{eqn:pole 1} and \eqref{eqn:pole 3} as exercises to the interested reader. \\

\noindent By \eqref{eqn:pole 2}, the residue of expressions \eqref{eqn:mi}--\eqref{eqn:pi} at $y = xp^{-2k'}$ is equal to:
$$
- \underset{z = xp^2}{\text{Res}} \left[ (1\bullet)  R_{\bullet 2} (p^{2k}) \bW^{(k-1)}_\bullet(z) R_{1\bullet} \left(\frac {xp^2}z \right) R_{21} \left( p^{-2} \right) \right. 
$$
$$
\left. (12) R_{21}(p^2) \bW_2^{(k'+1)} \left( \frac x{p^{2k'}} \right) \prod_{i=2}^{k-1} f \left( p^{-2i} \right) \right]
$$
If we move the permutation $(12)$ all the way to the left-hand side of the formula (and replace the product $R_{12}(p^{-2}) R_{21}(p^2)$ by $f(p^{-2})$), the expression above equals:
$$
- (12) \underset{z = xp^2}{\text{Res}} \left[ (2\bullet) R_{\bullet 1} (p^{2k}) \underbrace{\bW^{(k-1)}_\bullet(z) R_{2\bullet} \left(\frac {xp^2}z \right) \bW_2^{(k'+1)} \left( \frac x{p^{2k'}} \right)}_{(q-q^{-1})(2\bullet) W_2^{(k+k')}(xp^{-2k'})} \prod_{i=1}^{k-1} f \left( p^{-2i} \right) \right]
$$
The underbraced term has residue equal to the expression underneath it due to the induction hypothesis of \eqref{eqn:residue} for the pair $(k-1,k'+1)$, hence the formula above matches \eqref{eqn:residue} for the pair $(k,k')$. Similarly, due to \eqref{eqn:pole 4}, the residue of expressions \eqref{eqn:mi}--\eqref{eqn:pi} at $y = zp^{2c}$ (with $c \in \{-k',...,-k'+k-2\}$) is equal to:
$$
- \underset{z = xp^2}{\text{Res}} \left[(1\bullet) (\bullet 2) \bW^{(k'+c)}_\bullet(z) R_{2\bullet} \left(p^{2k-2}\right) \right.
$$
$$
\left. \bW^{(k-c-1)}_2(zp^{2c})  R_{\bullet 2} \left( p^{-2c} \right) \squiggly{R_{1\bullet} \left(\frac {xp^2}z \right)} \bW^{(1)}_1(x) R_{21} \left( p^{2c+2} \right)   \prod_{i=1}^{k-2} f \left( p^{-2i} \right) \right]
$$
The term with the squiggly underline has residue at $z = xp^2$ equal to a multiple of the permutation matrix $(1\bullet)$, so it can be moved left-to-right by changing the indices of various factors along the way. Hence the expression above equals:
$$
- \underset{z = xp^2}{\text{Res}} \left[(1\bullet) (\bullet 2) \bW^{(k'+c)}_\bullet(z) R_{1\bullet} \left(\frac {xp^2}z \right)  \right.
$$
$$
\left. R_{21} \left(p^{2k-2}\right) \underbrace{\bW^{(k-c-1)}_2(xp^{2c+2})  R_{1 2} \left( p^{-2c} \right)  \bW^{(1)}_1(x) R_{21} \left( p^{2c+2} \right)}_{R_{12}(p^{2-2k}) \bW^{(1)}_1(x) R_{21} (p^{2k}) \bW^{(k-c-1)}_2(xp^{2c+2})} \prod_{i=1}^{k-2} f \left( p^{-2i} \right) \right]
$$
The underbraced term is equal to the expression underneath it due to the equality between the rational functions representing \eqref{eqn:un} and \eqref{eqn:doi} for the pair $(1,k-c-1)$. Therefore, the expression above equals:
$$
-\underset{z = xp^2}{\text{Res}} \left[(1\bullet) (\bullet 2) \underbrace{\bW^{(k'+c)}_\bullet(z) R_{1\bullet} \left(\frac {xp^2}z \right) \bW^{(1)}_1(x)}_{(1\bullet)(q-q^{-1})\bW_1^{(k'+c+1)}(x)} R_{21} (p^{2k}) \bW^{(k-c-1)}_2(xp^{2c+2}) \prod_{i=1}^{k-1} f \left( p^{-2i} \right) \right]
$$
The underbraced term has residue equal to the expression underneath due to \eqref{eqn:residue} for the pair $(1,k'+c)$. Since the expression above matches \eqref{eqn:residue} for the pair $(k,k')$, the induction step of Claim \ref{claim:precise 2} is complete. \\

\end{proof} 

\section{The first shuffle algebra}
\label{sec:shuf 1}

\subsection{} We will often write elements $X \in \End(V^{\otimes k})$ as $X_{1...k}$ so as to point out the set of indices of $X$. Letting $E_{ij} \in \End(V)$ be the usual elementary matrix, we have:
\begin{equation}
\label{eqn:basis}
X = \sum_{i_1,...,i_k,j_1,...,j_k} \gamma \cdot E_{i_1j_1} \otimes ... \otimes E_{i_kj_k}
\end{equation}
for certain coefficients $\gamma$. For any permutation $\sigma \in S(k)$, we write:
\begin{equation}
\label{eqn:conjugation}
\sigma X \sigma^{-1} = X_{\sigma(1)...\sigma(k)}
\end{equation}
where $\sigma \curvearrowright V^{\otimes k}$ by permuting the factors (therefore, the effect of conjugating \eqref{eqn:basis} by $\sigma$ is to replace the indices $i_{\sigma(1)},...,j_{\sigma(k)}$ by $i_1,...,j_k$). Moreover, we will write:
\begin{equation}
\label{eqn:pseudo sweedler}
X_{1...k} = X_{1...i} \otimes X_{i+1...k} \in \End(V^{\otimes i}) \otimes \End(V^{\otimes k-i}) \cong \End(V^{\otimes k})
\end{equation}
if we wish to set apart the first $i$ tensor factors from the last $k-i$ tensor factors of $X$. There is an implicit summation in the right-hand side of \eqref{eqn:pseudo sweedler} which we will not write down, much alike Sweedler notation. For any $a \in \BN$, we will write:
$$
E_{ij}^{(a)} = 1 \otimes ... \otimes \underbrace{E_{ij}}_{a\text{--th position}} \otimes \ ... \otimes 1 \in \End(V^{\otimes k})
$$
(the number $k \geq a$ will always be clear from context). More generally, for any $X \in \End(V^{\otimes k})$ and any collection of distinct natural numbers $a_1,...,a_k$, write:
$$
X_{a_1...a_k} \in \End(V^{\otimes N})
$$
(the number $N \geq a_1,...,a_k$ will always be clear from context) for the image of $X$ under the map $\End(V^{\otimes k}) \rightarrow \End(V^{\otimes N})$ that sends the $i$--th factor of the domain to the $a_i$--th factor of the codomain, and maps to the unit in all factors $\neq \{a_1,...,a_k\}$. \\

\subsection{} 

Consider a parameter $\oq$ and denote $p = q^n \oq$. Then define:
\begin{align}
&\tR^+(x) = R_{21} \left( \displaystyle \frac 1{x\oq^2} \right) \qquad \qquad \in \End_{\BQ(q,\oq)}(V \otimes V)(x) \label{eqn:tr plus} \\ 
&\tR^-(x) = D_2 R_{21} \left( \displaystyle \frac {p^2}x \right) D_2^{-1} \quad \ \in \End_{\BQ(q,\oq)}(V \otimes V)(x) \label{eqn:tr minus}
\end{align}
where $D = \text{diag}(q^2,...,q^{2n}) \in \End(V)$. \\

\begin{definition}

For either choice of the sign $\pm$, consider the vector space:
\begin{equation}
\label{eqn:shuf new}
\bigoplus_{k = 0}^\infty \emph{End}_{\fff}(\underbrace{V \otimes ... \otimes V}_{k \text{ factors}})(z_1,...,z_k)^{\emph{Sym}}
\end{equation}
\footnote{The superscript ``Sym" refers to the vector subspace of symmetric tensors $X_{1...k}(z_1,...,z_k)$, i.e. those which satisfy the identity: 
\begin{equation}
\label{eqn:symmetry identity}
X_{1...k}(z_1,...,z_k) = R_\sigma \cdot X_{\sigma(1)...\sigma(k)} (z_{\sigma(1)},...,z_{\sigma(k)}) \cdot R_\sigma^{-1}
\end{equation}
for all permutations $\sigma$, where $R_\sigma$ is any braid lift of the permutation $\sigma$ (see Figure 1).} and endow it with an associative algebra structure, by setting $X*Y$ equal to:
\begin{equation}
\label{eqn:shuf prod}
\sum^{a_1<...<a_k, \ b_1<...<b_l}_{\{1,...,k+l\} = \{a_1,...,a_k\} \sqcup \{b_1,...,b_l\}} \left[ \prod_{i=k}^{1} \prod_{j=1}^l \underbrace{R_{a_ib_j} \left( \frac {z_{a_i}}{z_{b_j}} \right)}_{\text{only if }a_i < b_j} \right]
\end{equation}
$$
X_{a_1...a_k}(z_{a_1},...,z_{a_k}) \left[ \prod_{i=1}^k \prod_{j=l}^{1} \tR^\pm_{a_ib_j}  \left( \frac {z_{a_i}}{z_{b_j}} \right) \right] Y_{b_1...b_l}(z_{b_1},...,z_{b_l})  \left[  \prod_{i=k}^{1} \prod_{j=1}^l \underbrace{R_{a_ib_j} \left( \frac {z_{a_i}}{z_{b_j}} \right)}_{\text{only if }a_i > b_j} \right] 
$$
The operation $X * Y$ will be called the ``shuffle product". \\
	
\end{definition}

\subsection{} There is a pictorial way to represent the shuffle product above. To this end, elements of $\text{End}(V^{\otimes k})(z_1,...,z_k)$ (which will henceforth be called matrices) will be represented by braids on $k$ strands: every strand carries a label from $1,...,k$ that represents one of the tensor factors of $V^{\otimes k}$, and is decorated by one of the variables $z_1,...,z_k$. For example, the $R$--matrices:
$$
R_{ab} \left( \frac {z_a}{z_b} \right) \quad \text{and} \quad R_{ba}^{-1} \left(\frac {z_b}{z_a} \right)
$$
(for any indices $a,b \in \{1,...,k\}$) are represented by the two braids in Figure 1. \\

\begin{figure}[h]
\centering
\includegraphics[scale=0.3]{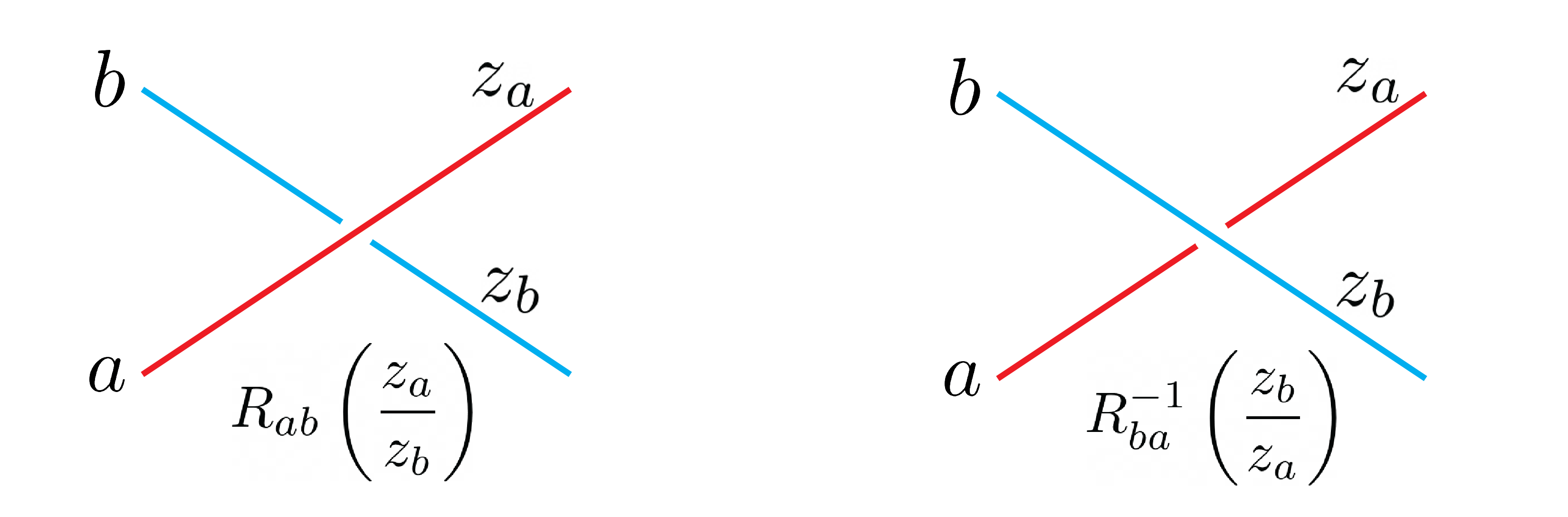} 
\caption{The $R$--matrices as crossings}
\end{figure}

\noindent The Yang-Baxter equation \eqref{eqn:ybe} corresponds to a Reidemeister III move: this is a statement about isotopic braids representing the same matrix. We will therefore feel free to move braids around in our pictures, without changing the particular element of $\text{End}(V^{\otimes k})(z_1,...,z_k)$ that they represent. Similarly, changing a crossing such as in Figure 1 in a braid has the effect of multiplying the corresponding matrix by the scalar-valued rational function $f(z_a/z_b)$, due to \eqref{eqn:unitary}. \\

\noindent The usual matrix multiplication of elements in $\text{End}(V^{\otimes k})(z_1,...,z_k)$ is represented as left-to-right concatenation of braids. Meanwhile, the summand of the shuffle product \eqref{eqn:shuf prod} corresponding to a fixed collection of indices $a_1<...<a_k$ and $b_1<...<b_l$ can be represented by the braid: 

\begin{figure}[h]
	\centering
	\includegraphics[scale=0.3]{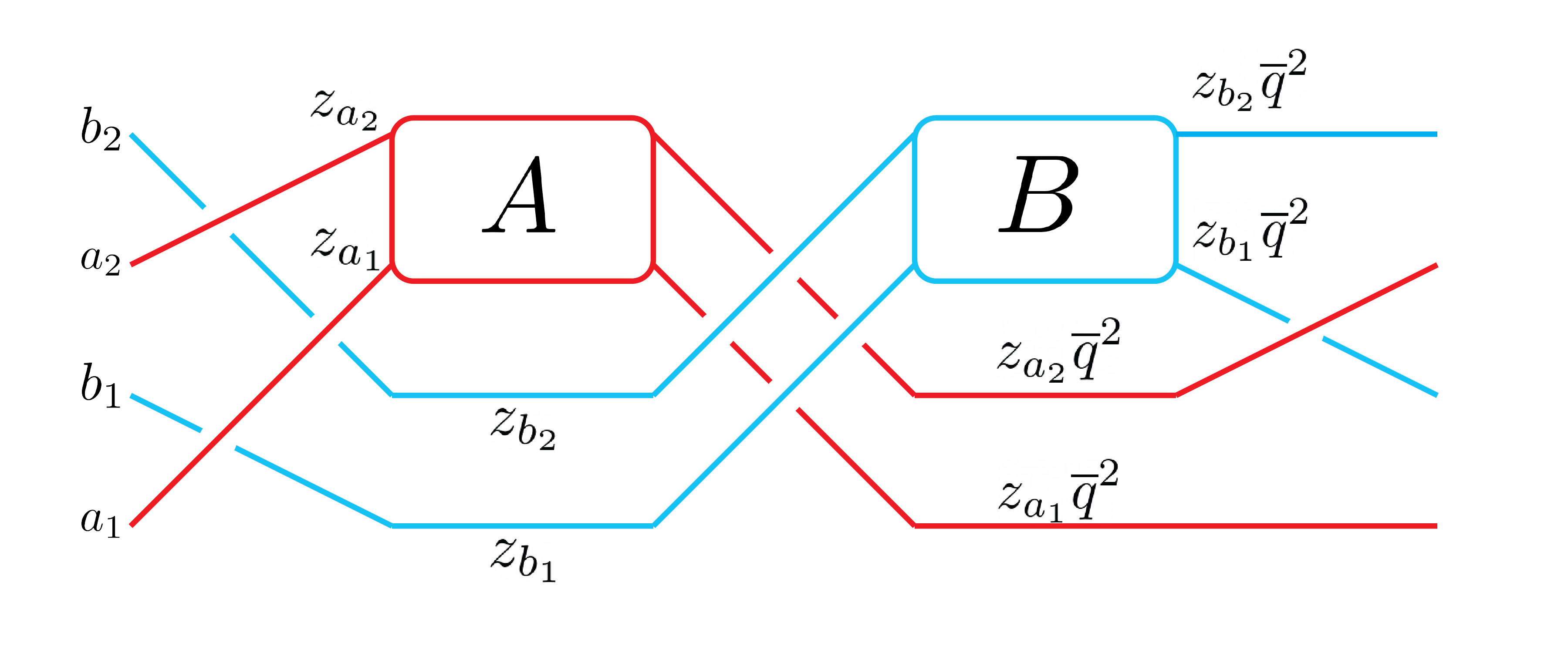} 
	\caption{The shuffle product as braids}
\end{figure}

\noindent (we depicted the case $\pm = +$, $(a_1,a_2) = (1,3)$ and $(b_1,b_2) = (2,4)$) in Figure 2). We make the convention that the variable on a strand does not change, except at a box. \\

\subsection{} For any permutation $\sigma \in S(k)$, we will consider the corresponding permutation operator $\sigma \in \End(V^{\otimes k})$. Letting $(ab)$ denote the transposition of $a$ and $b$, we will consider the following braid, called a \textbf{colon}:

\begin{figure}[h]
	\centering
	\includegraphics[scale=0.3]{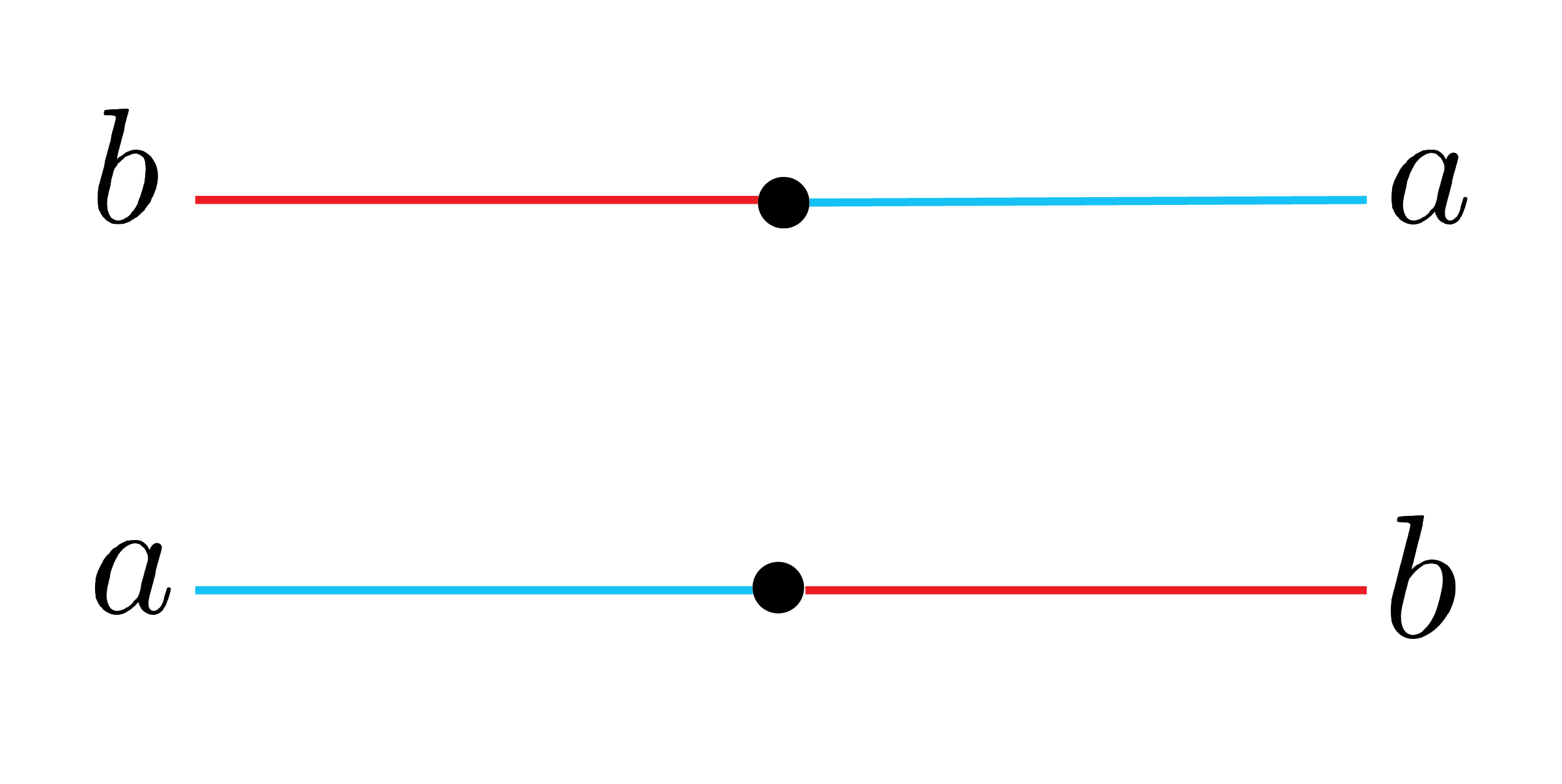} 
	\caption{A permutation braid}
\end{figure}

\noindent which we stipulate denotes the following multiple of a permutation operator:
\begin{equation}
\label{eqn:permutation}
(q^{-1} - q)  (ab)\in \End(V^{\otimes k})
\end{equation}
This operator is relevant to us because it matches the residue:
\begin{equation}
\label{eqn:residues}
\underset{x = 1}{\text{Res }} R_{ab}(x)  = (q^{-1} - q)(ab)
\end{equation}
We note that the label $a,b$ of the two strands in Figure 3 changes when encountering the colon, but the variable on each strand does not. This is the case in order to ensure identities such as:
$$
(ab) R_{ac} \left( \frac xy \right) = R_{bc} \left( \frac xy \right) (ab)
$$

\subsection{} 

For any matrix-valued rational function $X(z_1,...,z_k)$ with at most simple poles at $\{z_i - z_j \oq^2\}_{1 \leq i < j \leq k}$, we let:
\begin{equation}
\label{eqn:def residue}
\underset{\{z_1 = y, z_2 = y \oq^2,..., z_i = y \oq^{2(i-1)}\}}{\Res} X
\end{equation}
be the rational function in $y,z_{i+1},...,z_k$ obtained by successively taking the residue at $z_2 = z_1 \oq^2$, then at $z_3 = z_1 \oq^4$,..., then at $z_i = z_1 \oq^{2(i-1)}$ and finally relabeling the variable $z_1 \leadsto y$. More generally, for any collection of natural numbers:
\begin{equation}
\label{eqn:composition}
1 = c_1 < c_2 < ... < c_u < c_{u+1} = k + 1
\end{equation}
we will write:
\begin{equation}
\label{eqn:iterated}
\underset{\{z_{c_s} = y_s, z_{c_s+1} = y_s \oq^2, ..., z_{c_{s+1}-1} = y_s \oq^{2(c_{s+1}-c_s - 1)} \}_{\forall s \in \{1,...,u\}}}{\Res} X
\end{equation}
for the rational function in $y_1,..,y_u$ obtained by applying the iterated residue \eqref{eqn:def residue} to the groups of variables indexed by $\{c_1,c_1+1,...,c_2-1\}$, ..., $\{c_u,c_u+1,...,c_{u+1}-1\}$. \\

\begin{definition}
\label{def:shuf aff}
	
Let $\CA^+$ be the subset of \eqref{eqn:shuf new} consisting of matrix-valued rational functions with at most simple poles at $z_i - z_j \oq^2$, such that for any sequence \eqref{eqn:composition}, the iterated residue \eqref{eqn:iterated} is of the form depicted in Figure 4, for some $X^{(\lambda_1,...,\lambda_u)} \in \eEnd(V^{\otimes u})(y_1,...,y_u)$, where $\lambda_s = c_{s+1}-c_s$ for all $s \in \{1,...,u\}$. \footnote{The symbol ``blue over blue" is an artifact of \cite{Tale}, and it refers to the rational function:
$$
\left[ \prod_{1 \leq s<s' \leq u} \prod_{i=1}^{\lambda_s-1} \prod_{i' = 1}^{\lambda_{s'}-1} f \left(\frac {y_s \oq^{2i}}{y_{s'}\oq^{2i'}} \right) \right] \left[ \prod_{s=1}^u \prod_{1 \leq i < i' < \lambda_s} f(\oq^{2(i-i')}) \right]
$$} \\

\begin{figure}[ht]    
	\centering
	\includegraphics[scale=0.3]{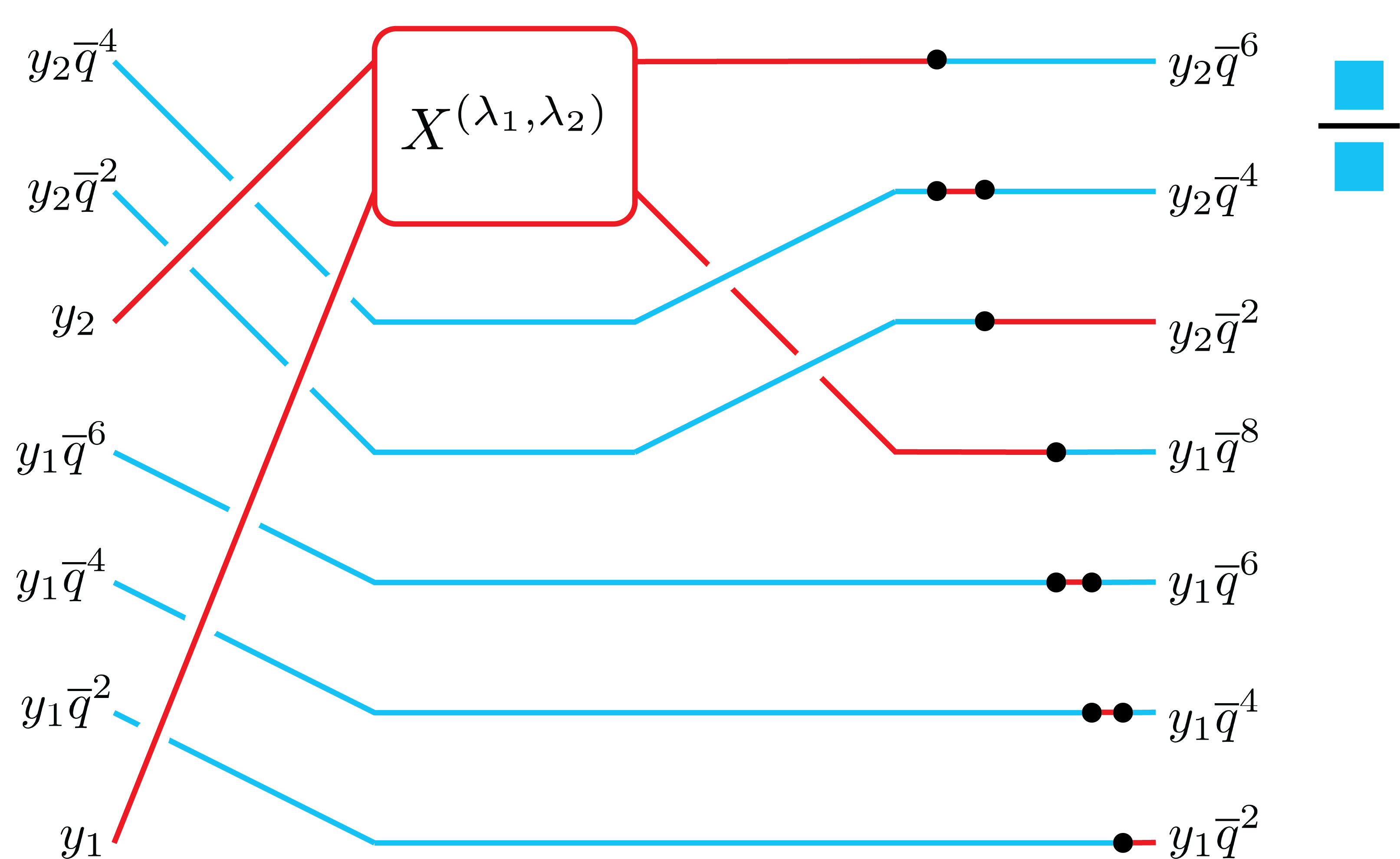}
   \caption{The residue conditions (for $u=2$, $\lambda_1 = 4$, $\lambda_2 = 3$)}
\end{figure}
	
\end{definition}

\noindent We define $\CA^-$ by analogy with $\CA^+$, but with $p^{-1} = q^{-n} \oq^{-1}$ instead of $\oq$, and with the diagonal matrix $D_1...D_k$ placed in front of the braid depicted in Figure 4. It was shown in \cite{Tale} that $\CA^\pm$ are subalgebras of the vector space \eqref{eqn:shuf new}, defined with respect to the product \eqref{eqn:shuf prod}. We will call either of them the \textbf{shuffle algebra}. \\

\subsection{} As in \eqref{eqn:notation 0}, we will write $E_{ij}^{(a)} := E^{(a)}_{\bari \barj} z_a^{\left \lfloor \frac {i-1}n \right \rfloor - \left \lfloor \frac {j-1}n \right \rfloor}$, and more generally:
\begin{equation}
\label{eqn:notation}
E_{i_1j_1} \otimes ... \otimes E_{i_kj_k} := E_{\bari_1 \barj_1} z_1^{\left \lfloor \frac {i_1-1}n \right \rfloor - \left \lfloor \frac {j_1-1}n \right \rfloor} \otimes ... \otimes E_{\bari_k \barj_k} z_k^{\left \lfloor \frac {i_k-1}n \right \rfloor - \left \lfloor \frac {j_k-1}n \right \rfloor}
\end{equation}
With this in mind, we grade the algebra $\CA^\pm$ by the monoid $\BZ^n \times \pm \{0,1,2,...\}$, via:
\begin{equation}
\label{eqn:deg}
\deg X^\pm = (\bd, \pm k)
\end{equation}
where we write $X^\pm$ for the element of $\CA^\pm$ represented by a matrix-valued rational function $X$, which is a linear combination of matrices of the form:
\begin{equation}
\label{eqn:summand}
f(z_1,...,z_k) E_{i_1j_1} \otimes ... \otimes E_{i_kj_k}
\end{equation}
with $(\text{hom deg } f) \cdot (1,...,1) - \sum_{a=1}^k [i_a;j_a) = \bd$. We define:
$$
\text{slope }X^\pm = \frac {|\bd|}{\pm k}
$$
where $|(d_1,...,d_n)| = d_1+...+d_n$. We will refer to $\bd$ and $\pm k$ of \eqref{eqn:deg} as the horizontal and vertical degrees of $X^\pm$, respectively, and write:
\begin{align*} 
&\hdeg X^\pm = \bd \\
&\vdeg X^\pm = \pm k
\end{align*}
and $\CA_{\pm k}$ to be the subspace of $\CA^\pm$ consisting of elements of vertical degree $\pm k$. \\

\subsection{} Consider the symmetrization operator:
\begin{equation}
\label{eqn:symmetrization aff}
\sym \ X = \sum_{\sigma \in S(k)} R_\sigma \cdot X_{\sigma(1)...\sigma(k)}(z_{\sigma(1)},...,z_{\sigma(k)}) \cdot R_\sigma^{-1} 
\end{equation}
where $R_\sigma$ is the product of $R_{ij} \left(\frac {z_i}{z_j} \right)$ associated to any braid lift of $\sigma$. For instance:
\begin{equation}
\label{eqn:big omega}
R_{\omega_k}(z_1,...,z_k) = \prod_{i=1}^{k-1} \prod_{j=i+1}^k R_{ij} \left( \frac {z_i}{z_j} \right) 
\end{equation}
lifts the longest element of $S(k)$. Consider the matrix-valued rational functions:
\begin{align*}
&\bQ^+(x) = - q \sum_{1 \leq i, j \leq n} \frac {(x\oq^2)^{\delta_{i \leq j}}}{1-x\oq^2} E_{ij} \otimes E_{ji}  \\
&\bQ^-(x) = - q \sum_{1 \leq i, j \leq n} \frac {(xp^{-2})^{\delta_{i \leq j}}}{1-xp^{-2}} E_{ij} \otimes E_{ji}  \cdot q^{2(j-i)}
\end{align*}
which have (up to scalar) the same simple pole and residue as $\tR^\pm(x)$. Define:
\begin{align*} 
&F_{ij}^{(k)} = \sym \left( R_{\omega_k} \prod_{a=1}^k \left[\prod_{b=1}^{a-2} \tR^+_{ba} \left( \frac {z_b}{z_a} \right) \bQ^+_{a-1,a} \left( \frac {z_{a-1}}{z_a} \right) E^{(a)}_{t_{a-1}t_a} \oq^{\frac {2\oo{t_a}}n} \right] \right) \in \CA^+ \\
&F_{ij}^{(-k)} = \frac 1{q^{2k}} \cdot \sym \left( R_{\omega_k} \prod_{a=1}^k \left[ \prod_{b=1}^{a-2} \tR^-_{ba} \left( \frac {z_b}{z_a} \right) \bQ^-_{a-1,a} \left( \frac {z_{a-1}}{z_a} \right) E^{(a)}_{t_{a-1}t_a} \oq^{-\frac {2\oo{t_{a-1}}}n} \right] \right) \in \CA^-
\end{align*}
\footnote{In the notation of \cite{Tale}, $F_{ij}^{(\pm k)} = \bF_{[j;i)}^{(\pm k)} \left(-p^ \frac 2n \right)^{\pm k}$} for all $(i,j) \in \zzz$ and $k \in \BN$, where:
$$
t_a = i + \left \lceil \frac {(j-i)a}k \right \rceil 
$$
for all $a \in \{1,...,k\}$. If $\mu = \frac {i-j}k \in \BQ$, we will write: 
$$
F_{ij}^{\pm \mu} = F_{ij}^{(\pm k)} 
$$
Note that $\deg F_{ij}^{(\pm k)} = (-[i;j), \pm k) \in \zz \times \BZ$. We set $F_{ij}^\mu = 0$ if $i-j \notin \mu \BZ$. \\\\

\subsection{} The elements $F_{ij}^{\mu}$ are connected to the following slope filtration on $\CA^\pm$. \\

\begin{definition}
\label{def:slope}
	
For any $\mu \in \BQ$, we define $\CA_{\leq \mu}^\pm \subset \CA^\pm$ to be the subset of matrix-valued rational functions $X = X_{1...k}(z_1,...,z_k)$ which have the property that: 
$$
X = \sum^{\text{various }X',X''}_{\text{with slope }X'' \leq \mu} X'_{1...i}(z_1,...,z_i) \otimes X''_{i+1...k}(z_{i+1},...,z_k) 
$$
when expanded as $z_1,...,z_i \ll z_{i+1},...,z_k$, for all $i \in \{0,...,k\}$ (recall \eqref{eqn:pseudo sweedler}). We further let $\CB_\mu^\pm \subset \CA_{\leq \mu}^\pm$ denote the subset of elements $X$ whose slope equals $\mu$. \\
	
\end{definition}

\noindent It was shown in \cite{Tale} that $\CA^\pm_{\leq \mu}$ and $\CB^\pm_{\mu}$ are subalgebras, and that:
$$
F_{ij}^{\pm \mu} \in \CB_{\mu}^\pm
$$
whenever $\frac {i-j}{\mu} \in \BN$. Moreover, if $\mu = \frac ab$ with $\gcd(a,b) = 1$, there is an isomorphism:
$$
\CB_\mu \cong U_q(\dot{\fgl}_{\frac ng})^{\otimes g}
$$
where $g = \gcd(n,a)$, under which the various $F_{ij}^{\pm \mu}$ are sent to the root generators of quantum groups (see Section 3 of \cite{Tale}). Linear bases of $\CA^\pm$ are given by:
\begin{equation}
\label{eqn:linear basis 1} 
\CA^\pm = \bigoplus_{\frac {i_1-j_1}{k_1} \geq ... \geq \frac {i_t-j_t}{k_t}} \fff \cdot F_{i_1j_1}^{(\pm k_1)} ... F_{i_tj_t}^{(\pm k_t)} 
\end{equation}
where the sum goes over all $(i_a,j_a) \in \zzz$ and $k_a \in \BN$. We define the completion:
\begin{equation}
\label{eqn:completion}
\widehat{\CA}^+ = \bigoplus_{k,d \in \BZ}  \mathop{\prod^{i_1-j_1+...+i_t-j_t = d}_{\frac {i_1-j_1}{k_1} \geq ... \geq \frac {i_t-j_t}{k_t}}}^{k_1+...+k_t = k} \fff \cdot F_{i_1j_1}^{(k_1)} ... F_{i_tj_t}^{(k_t)}
\end{equation}
In other words, an element of the completion $\widehat{\CA}^+$ is an infinite linear combination of the elements $F_{i_1j_1}^{(k_1)} ... F_{i_tj_t}^{(k_t)}$, but with $k_1+...+k_t$ and $i_1-j_1+...+i_t-j_t$ taking finitely many values. As in Proposition 3.7 of \cite{W surf}, one can show that this completion is closed under multiplication, and thus a well-defined algebra. \\ 

\subsection{} The slope filtration can be expressed in terms of the extended algebras:
\begin{align*}
&\tCA^+ = \CA^+ \otimes \uug \Big / \text{relations} \\
&\tCA^- = \CA^- \otimes \uul \Big / \text{relations}
\end{align*}
(the explicit relations above can be found in \cite{Tale}, although we will not need them). We recall the generating series \eqref{eqn:series s} and \eqref{eqn:series t} of elements of the algebras $\uug$ and $\uul$, respectively. We will henceforth use the notation:
\begin{align*}
&T^-(x) = T(x) \\ 
&S^-(x)  = T(x \oq^2)^{-1}
\end{align*}
for elements of $\tCA^-[[x]] \supset \uul[[x]]$, and:
\begin{align*}
&S^+(x) = S(x) \\
&T^+(x) = D \left[ \left( S (x p^2)^\dagger \right)^{-1} \right]^\dagger D^{-1}
\end{align*}
for elements of $\tCA^+[[x^{-1}]] \supset \uug[[x^{-1}]]$, where $D = \text{diag}(q^2,...,q^{2n})$ and $\dagger$ denotes transposition of matrices. This notation will allow us to define the following coproducts on the extended algebras $\tCA^\pm$:
\begin{equation}
\label{eqn:coproduct}
\Delta \left(X^\pm_{1...k}(z_1,...,z_k) \right) = \sum_{i=0}^k
\end{equation}
$$
(S_{k}^\pm(z_k)...S^\pm_{i+1}(z_{i+1}) \otimes 1) \left[ \frac {X_{1...i} \left( z_1,...,z_i \right) \otimes X_{i+1...k} \left(z_{i+1}, ..., z_k \right)}{\prod_{1\leq u \leq i < v \leq k} f \left( \frac {z_u}{z_v} \right)} \right]^\pm (T^\pm_{i+1}(z_{i+1}) ... T^\pm_{k}(z_k) \otimes 1) 
$$
\footnote{As explained in \cite{Tale}, the expression in square brackets is defined by taking the rational function: 
$$
\frac {X_{1...k}(z_1,...,z_k)}{\prod_{1\leq u \leq i < v \leq k} f \left( \frac {z_u}{z_v} \right)}
$$ separating the variables $z_1,...,z_i$ from $z_{i+1},...,z_k$ by expanding in the range $|z_1|,...,|z_i| \ll |z_{i+1}|,...,|z_k|$, and then separating the tensor $X_{1...k}$ into $X_{1...i} \otimes X_{i+1...k}$ according to \eqref{eqn:pseudo sweedler}.} By virtue of Definition \ref{def:slope}, we have: 
\begin{align} 
&X^\pm \in \CA^\pm_{\leq \mu} \quad \Rightarrow \quad \Delta(X^\pm) = \sum^{\text{various } X',X''}_{\text{slope } X'' \leq \mu} X' \otimes X'' \label{eqn:obs 1} \\
&X^\pm \in \CB^\pm_{\mu} \quad \Rightarrow \quad  \Delta(X^\pm) = \sum^{\text{various } X',X''}_{\text{slope } X'' \leq \mu \leq \text{slope } X'} X' \otimes X'' \label{eqn:obs 2}
\end{align}
The latter observation allows us to define a coproduct $\Delta_\mu$ on the subalgebra $\CB_\mu$ by taking the leading slope summands in the coproduct $\Delta$:
\begin{equation}
\label{eqn:copoduct mu}
\Delta_\mu \left(X^\pm\right) = \sum^{\text{those summands in \eqref{eqn:obs 2}}}_{\text{with }\text{slope } X'' = \mu = \text{slope } X'} X' \otimes X''
\end{equation}
For example, the coproduct of the elements $F_{ij}^\mu$ satisfy the following formula:
\begin{equation}
\label{eqn:coproduct f}
\Delta_\mu \left( F^\mu_{ij} \right) = \sum_{\bullet \in \{i,...,j\}} \frac {\psi_\bullet}{\psi_i} F^\mu_{\bullet j} \otimes F^\mu_{i \bullet}
\end{equation}
where $\psi_k \in \uug \subset \tCA^+$ and $\psi_k \in \uul \subset \tCA^-$ are defined in \eqref{eqn:psi}. \\
 
\subsection{} Finally, we recall a distinguished pairing between the algebras $\CA^\pm$. Before recalling its definition, we will need an important fact about these algebras, namely that they are generated by the elements \eqref{eqn:notation} in degree 1. \\

\begin{proposition}
\label{prop:generation}

(\cite{Tale}) The algebras $\CA^\pm$ are generated by $\left\{ F_{ij}^{(\pm 1)} \right\}_{(i,j) \in \zzz}$. \\

\end{proposition}

\noindent In other words, any element of $\CA^\pm$ can be written as a linear combination of:
\begin{align}
&I_1^+ * ... * I^+_k \in \CA^+ \label{eqn:linear 1} \\ 
&J_1^- * ... * J^-_k \in \CA^- \label{eqn:linear 2}
\end{align}
respectively, as $I_1,...,I_k$ (respectively $J_1,...,J_k$) range over $\End(V)[z_1^{\pm 1}]$. \\

\begin{proposition}
\label{prop:pairing}

(\cite{Tale}) There is a non-degenerate pairing: 
\begin{equation}
\label{eqn:pairing a}
\tCA^+ \otimes \tCA^- \rightarrow \fff
\end{equation}
which satisfies the following formulas for all $I_a, J_a \in \eEnd(V)[z^{\pm 1}]$, $X^\pm \in \CA^\pm$:
\begin{equation}
\label{eqn:symm pair 1} 
\Big \langle I^+_1 * ... * I^+_k, X^-_{1...k}(z_1,...,z_k) \Big \rangle = (q^2-1)^k \int_{|z_1| \ll ... \ll |z_k|} 
\end{equation}
$$
\emph{Tr}_{V^{\otimes k}} \left( R_{\omega_k} \prod_{a=1}^k \left[ I_a^{(a)}(z_a) \prod_{b=a+1}^k \tR^+_{ab} \left(\frac {z_a}{z_b} \right) \right] \frac {X_{1...k}(z_1,...,z_k)}{\prod_{1 \leq i < j \leq k} f\left(\frac {z_i}{z_j} \right)} \right) 
$$
\begin{equation}
\label{eqn:symm pair 2} 
\Big \langle X^+_{1...k}(z_1,...,z_k), J^-_1 * ... * J^-_k \Big \rangle = (q^2-1)^k \int_{|z_1| \ll ... \ll |z_k|} 
\end{equation}  
$$
\emph{Tr}_{V^{\otimes k}} \left( R_{\omega_k} \prod_{a=1}^k \left[ J_a^{(a)}(z_a) \prod_{b=a+1}^k \tR^-_{ab} \left(\frac {z_a}{z_b} \right) \right] \frac {X_{1...k}(z_1,...,z_k)}{\prod_{1 \leq i < j \leq k} f\left(\frac {z_i}{z_j} \right)} \right)
$$
The pairing of elements of non-opposite vertical degrees is defined to be 0. \footnote{Above, the notation: 
$$
\int_{|z_1| \ll ... \ll |z_k|} F(z_1,...,z_k)
$$ 
refers to the iterated residue of $F$ at $0$: first in $z_1$, then in $z_2$,..., finally in $z_k$. For brevity, we do not write the differential $\prod_{a=1}^k \frac {dz_a}{2\pi i z_a}$ at the end of all our integrands, though it is always implied.} Moreover, \eqref{eqn:pairing a} is a bialgebra pairing, in that the dual of the product is the coproduct:
\begin{align}
&\langle ab, c \rangle = \langle b \otimes a, \Delta(c) \rangle, \quad \forall a, b \in \tCA^+, c \in \tCA^- \label{eqn:bialg 1} \\
&\langle a, bc \rangle = \langle \Delta (a), b \otimes c \rangle, \qquad \forall a \in \tCA^+, b, c \in \tCA^- \label{eqn:bialg 2}
\end{align}
	
\end{proposition}

\noindent When we restrict the pairing \eqref{eqn:pairing a} to the subalgebras $\CB_\mu$, it is also true that the dual of the product is the coproduct $\Delta_\mu$, because the pairing of any two elements is trivial unless they have opposite degrees. \\

\begin{proposition}
\label{prop:pair 1}

If $x^+ = \prod^\leftarrow_{\mu \in \BQ} x^+_{\mu}$, $x^- = \prod^\rightarrow_{\mu \in \BQ} x^-_{\mu}$ with $x^\pm_{\mu} \in \CB^\pm_\mu$ (the product labeled by $\leftarrow$ is taken in decreasing order of $\mu$, and the one labeled by $\rightarrow$ is taken in increasing order of $\mu$, and all but finitely many of the $x^\pm_\mu$'s are equal to 1), then:
\begin{equation}
\label{eqn:pairing multiplicative}
\langle x^+, x^- \rangle = \prod_{\mu \in \BQ} \langle x^+_\mu, x^-_\mu \rangle 
\end{equation}

\end{proposition}

\begin{proof} We will prove the required statement by induction over the total number of elements among the $x_\mu^\pm$'s which are not equal to 1. Let us assume $\mu = \alpha$ is the greatest index for which one of the $x_\mu^\pm$'s is not equal to 1, and let us assume without loss of generality that $|\vdeg x_\alpha^+| \geq |\vdeg x_\alpha^-|$. Then \eqref{eqn:bialg 1} implies that:
$$
\langle x^+, x^- \rangle = \left \langle  x_\alpha^+ \prod^\leftarrow_{\mu < \alpha} x_\mu^+, \prod^\rightarrow_{\mu \leq \alpha} x_\mu^- \right \rangle = \left \langle  \prod^\leftarrow_{\mu < \alpha} x_\mu^+ \otimes x_\alpha^+, \prod^\rightarrow_{\mu \leq \alpha} \Delta(x_\mu^-) \right \rangle
$$
The fact that $x_\mu^- \in \CA_{\leq \mu}^- \subset \CA_{\leq \alpha}^-$ implies that the second tensor factors of all $\Delta(x_\mu^-)$ have slope $\leq \alpha$, and equality only holds for $x_\alpha^-$ itself. We have:
\begin{equation}
\label{eqn:Delta}
\Delta(x_\alpha^-) = (\text{product of }\psi_k^{\pm 1}\text{'s}) \otimes x_\alpha^- + x_\alpha^- \otimes 1 + ...
\end{equation}
where all the summands marked by the ellipsis either have second tensor factor of slope $<\alpha$, or slope $=\alpha$ but $|\vdeg|$ strictly smaller than that of $x_\alpha^-$. Hence the only summand which pairs non-trivially with $x_\alpha^+$ is the first one of \eqref{eqn:Delta}, so we have:
$$
\langle x^+, x^- \rangle = \left \langle \prod^\leftarrow_{\mu < \alpha} x_\mu^+, \prod^\rightarrow_{\mu < \alpha} x_\mu^- \right \rangle \langle x^+_\alpha, x^-_\alpha \rangle
$$
Repeating this argument finitely many times implies \eqref{eqn:pairing multiplicative}.
	
\end{proof}

\subsection{} For any $X^- \in \CA_{-k}$, recall $X^{(k)}(y) \in \End(V)[y^{\pm 1}]$ from Definition \ref{def:shuf aff}. \\

\begin{proposition}
\label{prop:pair f}

(\cite{Tale}) For any $(i,j) \in \zzz$ and $k \in \BN$, we have:
$$
\left \langle F_{ij}^{(k)}, X^- \right \rangle = \text{coefficient of } E_{\barj \bari} y^{\left \lfloor \frac {j-1}n \right \rfloor - \left \lfloor \frac {i-1}n \right \rfloor} \text{ in}
$$
$$ 
(q^2-1)^k p^{\frac {k(j-i)+k-(j-i)-\gcd(j-i,k)+2\bari(k-1) + 2\barj}n} X^{(k)}(y)
$$
for all $X \in \CB_{\mu}^-$. We recall that $p = q^n \oq$. \\

\end{proposition}

\noindent If we apply Proposition \ref{prop:pair f} for $X^- = F_{j'i'}^{(-k')}$, we obtain:
\begin{equation}
\label{eqn:one}
\left \langle F_{ij}^{(k)}, F_{j'i'}^{(-k')} \right \rangle = \delta_{(i',j')}^{(i,j)} \delta_{k'}^k (1-q^{-2}) 
\end{equation}
The reason for this is the fact that:
\begin{equation}
\label{eqn:syn}
\left(F_{j'i'}^{(-k')} \right)^{(k')}(y) = E_{\barj'\bari'} y^{\left \lfloor \frac {j'-1}n \right \rfloor - \left \lfloor \frac {i'-1}n \right \rfloor}  
\end{equation}
$$
q^{-2k'}(1-q^{-2})^{-k'+1} p^{\frac {-k'(j'-i')-k'+(j'-i')+\gcd(i'-j',k')-2(k-1)\bari'-2\barj'}n}
$$
as proved in Proposition 5.21 of \cite{Tale}. \\

\begin{proposition}
\label{prop:pair aff}

For any $\mu \in \BQ$ and indices $i,j, a_1,...,a_t,b_1,...,b_t \in \BZ$ such that $i-j, a_1-b_1,...,a_t-b_t \in \mu \BN$, we have:
\begin{equation}
\label{eqn:pair f}
\Big \langle F^\mu_{ij}, F^\mu_{a_1b_1} ... F^\mu_{a_tb_t} \Big \rangle = (1-q^{-2})^t
\end{equation}
if $b_1 \equiv a_2$, $b_2 \equiv a_3$, ..., $b_{t-1} \equiv a_t$, $\sum_{k=1}^t [a_k;b_k) = [j;i)$ and 0 otherwise. \\
	
\end{proposition}

\begin{proof} By \eqref{eqn:bialg 2}, the left-hand side of \eqref{eqn:pair f} is equal to:
$$
\left \langle \Delta^{(t-1)} \left( F^\mu_{ij} \right), F^\mu_{a_1b_1} \otimes ... \otimes F^\mu_{a_tb_t} \right \rangle = 
$$
\begin{align*} 
&\stackrel{\eqref{eqn:coproduct f}}= \sum_{j = c_0 , c_1,...,c_{t-1}, c_t = i} \left \langle \bigotimes_{k=1}^t \frac {\psi_{c_k}}{\psi_i} F^\mu_{c_kc_{k-1}}, \bigotimes_{k=1}^t F^\mu_{a_kb_k}  \right \rangle = \\
&\stackrel{\eqref{eqn:one}}= \sum_{j = c_0 , c_1,...,c_{t-1}, c_t = i} \prod_{k=1}^t \delta_{(c_{k-1},c_k)}^{(a_k,b_k)} (1-q^{-2})
\end{align*}
The expression above precisely matches the right-hand side of \eqref{eqn:pair f}. 	
	
\end{proof}

\subsection{} Let us now consider the following elements of $\wCA^+$: \\

\begin{definition}
\label{def:series}

For any $(i,j) \in \zzz$ and $k > 0$, define $W_{ij}^{(k)} \in \wCA^+$ by:
\begin{equation}
\label{eqn:def w}
W_{ij}^{(k)} = \sum_{k_1,...,k_t \in \BN}^{k_1+...+k_n = k} \sum^{i = c_0,c_1,...,c_{t-1},c_t = j}_{\frac {c_0-c_1}{k_1} > ... > \frac {c_{t-1}-c_t}{k_t}} F_{c_0c_1}^{(k_1)} ... F_{c_{t-1}c_t}^{(k_t)} \cdot p^{\frac {2(\alpha(v) - (k-1)\bari - \barj)}n} 
\end{equation}
where $p = q^n \oq$, and we associate the integer:
$$
\alpha(v) = \sum_{s=1}^t \left[ \frac {k_s(c_{s-1}-c_s) - k_s - (c_{s-1}-c_s) + \gcd(c_{s-1}-c_s,k_s)}2 + k_s(c_0-c_{s-1}) \right]
$$
to the sequence of lattice points $v = \{(c_0-c_1,k_1),...,(c_{t-1}-c_t,k_t)\}$. \\ 
 
\end{definition}

\noindent The abuse of notation between the symbols of Subsection \ref{sub:symbols} and the elements \eqref{eqn:def w} is intentional, since sending the former to the latter will give rise to the isomorphism of Theorem \ref{thm:main 1}. Let us form the generating series:
\begin{equation}
\label{eqn:def w series}
W^{(k)}(y) = \sum_{(i,j) \in \zzz} W_{ij}^{(k)} \cdot E_{ij} 
\end{equation}
which takes values in $\wCA^+ \otimes \End(V)[[y^{\pm 1}]]$, where $E_{ij} = E_{\bari\barj} y^{\left \lfloor \frac {i-1}n \right \rfloor - \left \lfloor \frac {j-1}n \right \rfloor}$. \\

\begin{proposition} 
\label{prop:pair w}

For any $X^- \in \CA_{-k}$, we have:
\begin{equation}
\label{eqn:pairing w}
\Big \langle W^{(k)}(y) , X^- \Big \rangle = (q^2-1)^k \cdot X^{(k)}\left( \frac 1y \right)^\dagger
\end{equation}
where $X^{(k)}(y)$ is the matrix featuring in Definition \ref{def:shuf aff}, and $\dagger$ denotes transpose. \\
	
\end{proposition} 

\noindent Let us compare Proposition \ref{prop:pair f} with Proposition \ref{prop:pair w}. The latter generalizes the formula in the former from $X^- \in \CB_\mu^-$ to all $X^- \in \CA^-$, but at the cost of replacing the single element $F_{ij}^\mu$ by the entire series \eqref{eqn:def w} of products of such elements. \\ 

\begin{proof} We will refer to elements $[i;j) \in \zz$ as ``arcs", in the sense that if we mark vertices $1,...,n$ on a circle, then $[i;j)$ corresponds to the arc starting at the vertex $\bari$ and going $j-i-1$ steps in the positive direction. We will say that arcs $[a;b)$ and $[a';b')$ string together to form the arc $[i;j)$ if $b \equiv a'$ and $[a;b) + [a';b') = [i;j)$. \\
	
\noindent Because formula \eqref{eqn:pairing w} is additive, it suffices to prove it for:
$$
X^- = \prod_{\mu}^\rightarrow \left[ F^{(-k_1^\mu)}_{a^\mu_1 b^\mu_1} F^{(-k_2^\mu)}_{a^\mu_{2} b^\mu_{2}} ... \right]
$$
(the product is taken in increasing order of $\mu \in \BQ$), for various choices of: 
$$
(a_s^\mu, b_s^\mu) \in \zzz \quad \text{and} \quad k_s^\mu \in \BN \quad \text{such that} \quad \frac {b_s^\mu - a_s^\mu}{k_s^\mu} = \mu 
$$
By Proposition \ref{prop:pair 1} and \eqref{eqn:pair f}, the left-hand side of \eqref{eqn:pairing w} vanishes for our $X^-$ unless the arcs $[a_1^\mu;b_1^\mu)$, $[a_2^\mu;b_2^\mu)$,... string together to form an arc $[j;i)$, so we have:
$$
\text{LHS of \eqref{eqn:pairing w}} = \begin{cases} p^{\frac {2(\alpha(v)-(k-1)\bari-\barj)}n} (1-q^{-2})^{\# \text{ arcs}} E_{ij} &\text{if the arcs string to form }[j;i) \\ 0 &\text{otherwise} \end{cases}
$$
where $v$ is the sequence $\{...,(b_2^\mu-a_2^\mu,k_2^\mu),(b_1^\mu-a_1^\mu,k_1^\mu),...\}$ in increasing order of $\mu$. Meanwhile, we have the multiplicativity property of the operation $X \leadsto X^{(k)}(y)$:
$$
(X*X')^{(k+k')}(y) = X^{(k)}(y) {X'}^{(k')}(y p^{-2k})
$$
for any $X \in \CA_{-k}$, $X' \in \CA_{-k'}$ (as proved in Proposition 5.12 of \cite{Tale}). Taking into account \eqref{eqn:syn}, we see that:
\begin{equation}
\label{eqn:above}
\left( \prod_{\mu}^\rightarrow F^{(-k_1^\mu)}_{a^\mu_1 b^\mu_1} F^{(-k_2^\mu)}_{a^\mu_{2} b^\mu_{2}} ...  \right)^{(k)}(y)
\end{equation}
vanishes unless the arcs $[a_1^\mu;b_1^\mu)$, $[a_2^\mu;b_2^\mu)$,... string together in increasing order of $\mu$ to form an arc $[j;i)$. More explicitly, we have:
$$
\text{RHS of \eqref{eqn:pairing w}} = \begin{cases} p^{\frac dn}  (1-q^{-2})^{\#\text{ arcs}} E_{ij} &\text{if the arcs string to form }[j;i) \\ 0 &\text{otherwise}  \end{cases} 
$$
where:
$$
d = \sum_\mu \left[ \sum_{s=1,2,...} k_s^\mu (b_s^\mu - a_s^\mu) - k_s^\mu - (b_s^\mu - a_s^\mu) + \gcd(b_s^\mu - a_s^\mu,k_s^\mu) - 2(k_s^\mu-1)\overline{b_s^\mu} - 2 \overline{a_s^\mu} \right] -
$$
$$
- \sum^{\mu < \mu' \text{ and any }s,s', \text{ or}}_{\mu = \mu' \text{ and } s<s'} 2 n k^\mu_s \left( \left \lfloor \frac {a_{s'}^{\mu'} - 1}n \right \rfloor - \left \lfloor \frac {b_{s'}^{\mu'} - 1}n \right \rfloor  \right)
$$
Given that the arcs $[a_1^\mu;b_1^\mu)$, $[a_2^\mu;b_2^\mu)$,... string together in increasing order of $\mu$ to form an arc $[j;i)$, it is an easy combinatorial exercise to see that $d = 2(\alpha(v)-(k-1)\bari-\barj)$. This shows that the left and right hand sides of \eqref{eqn:pairing w} are equal.

\end{proof}

\subsection{} We will now deduce quadratic relations between the generating series $W^{(k)}(y)$, and show that they match \eqref{eqn:w rels}. To this end, we recall from \cite{Tale} that for any natural numbers $k \leq k'$ and $X^- \in \CA_{-k-k'}$, the rational function:
$$
X^{(k,k')}(x,y)
$$
has at most simple poles at:
$$
y = x p^{-2a} \qquad \text{for } a \in \{1,...,k\} \sqcup \{-k',...,k-k'-1\} 
$$
The Proposition below computes the corresponding residues: \\

\begin{proposition}
\label{prop:residues}

Consider any natural numbers $k \leq k'$ and any $X^- \in \CA_{-k-k'}$. For any $a \in \{1,...,k\}$, we have:
\begin{equation}
\label{eqn:res 1}
\underset{y = x p^{-2a}}{\emph{Res}} X^{(k,k')}(x,y) = 
\end{equation}
$$
= R \left(p^{2a} \right) X_{21}^{(k-a,k'+a)}(xp^{-2a},x) (q^{-1} - q) (12) \prod_{i=1}^{a-1} f(p^{-2i}) 
$$
while for any $a \in \{-k',...,k-k'-1\}$, we have:
\begin{equation}
\label{eqn:res 2}
\underset{y = x p^{-2a}}{\emph{Res}} X^{(k,k')}(x,y) = 
\end{equation}
$$ 
= X_{12}^{(k'+a,k-a)}(x,xp^{-2a})  R \left( p^{2(k-k'-a)} \right) (q - q^{-1}) (12) \prod_{i=1}^{k-k'-a-1} f(p^{-2i}) 
$$
\footnote{To make formulas \eqref{eqn:res 1} and \eqref{eqn:res 2} hold in the cases $a=k$ and $a=-k'$, we declare: 
$$
X_{21}^{(0,k+k')}(xp^{-2k},x) = X_1^{(k+k')}(x) \qquad \text{and} \qquad X_{12}^{(0,k+k')}(x,xp^{2k'}) = X_2^{(k+k')}(xp^{2k'})
$$} We note that the right-hand sides of all expressions above are regular. \\

\end{proposition}

\begin{proof} We will first prove \eqref{eqn:res 1}, and note that our proof only requires us to assume that $k' + a > k > a > 0$ (instead of the stronger condition $k' \geq k$). By definition:
\begin{equation}
\label{eqn:collection}
\underset{z_1 = x, z_2 = x p^{-2},...,z_k = x p^{2-2k}, z_{k+1} = x p^{-2a}, z_{k+2} = x p^{-2a-2},..., z_{k+k'} = x p^{2-2a-2k'}}{\res} X	
\end{equation}
(henceforth referred to simply as ``the residue of $X$") is given by the braid in Figure 4, with $X^{(\lambda_1,\lambda_2)}$ replaced by the left-hand side of \eqref{eqn:res 1}. Throughout the present proof, we will ignore the diagonal matrix $D_1...D_{k+k'}$ and the factors ``blue over blue" that are required to multiply the braid in Figure 4, because they will be the same in all the braids we depict, and hence will not affect our proof. \\

\noindent The symmetry property \eqref{eqn:symmetry identity} implies that the braid that represents the residue of $X$ is unchanged if we conjugate it with any product of crossings. In the figure below, we make a particular choice of such a product of crossings, and conclude that the residue of $X$ is also represented by the following braid (note that the color coding of the red, blue and purple strands doesn't mean anything other than the fact that they will play a special role in the argument below, and having different colors will make the proof easier to follow): 
	
\begin{figure}[ht]    
\centering
\includegraphics[scale=0.2]{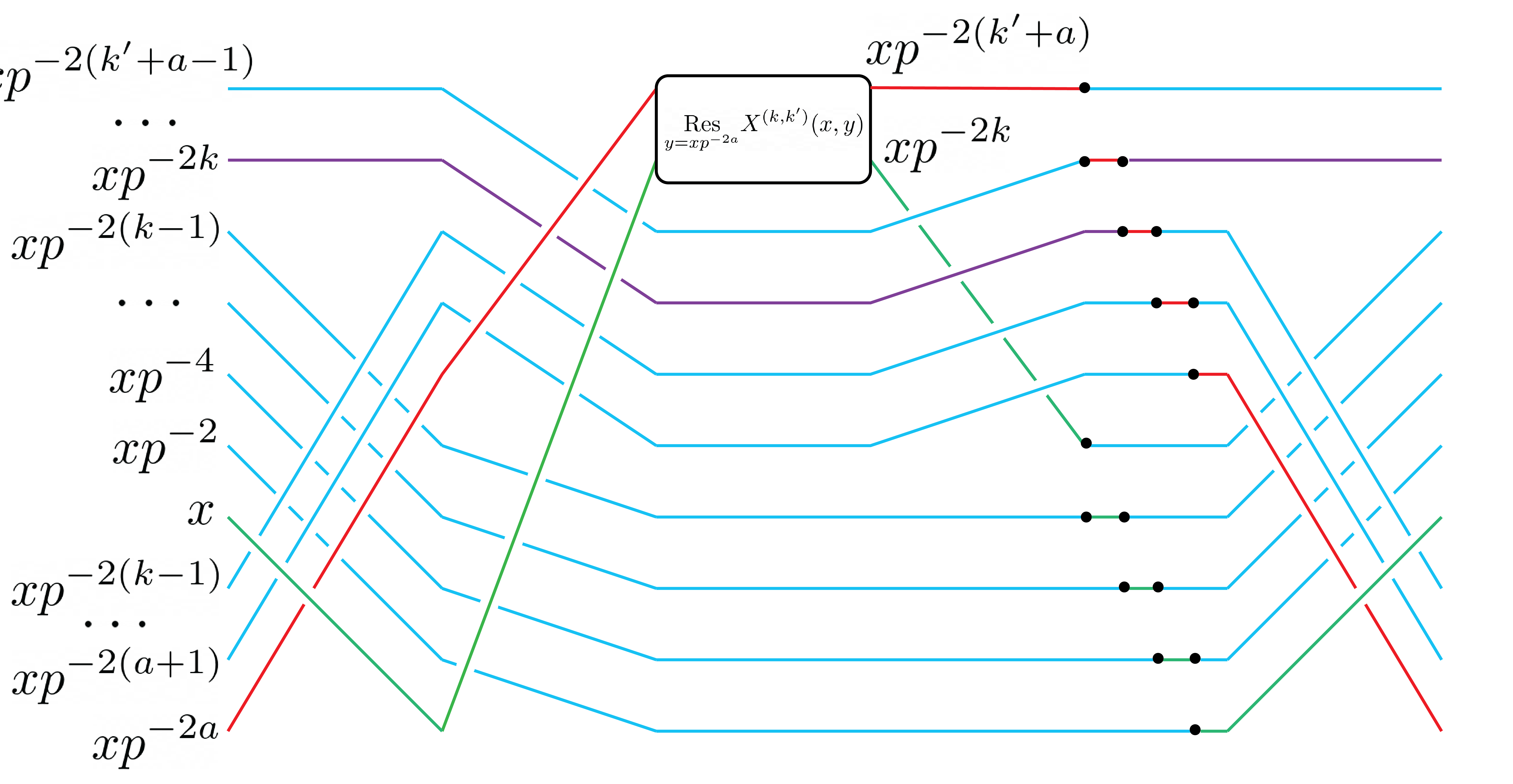}
\end{figure}

\noindent In the formula above, let us now replace:
\begin{equation}
\label{eqn:substitution}
\underset{y = x p^{-2a}}{\text{Res}} X^{(k,k')}(x,y) = R \left(p^{2a} \right) \cdot Y \cdot (q^{-1} - q) (12) 
\end{equation}
and the goal will be to prove:
\begin{equation}
\label{eqn:finally}
Y = X_{21}^{(k-a,k'+a)}(xp^{-2a},x) \prod_{i=1}^{a-1} f(p^{-2i}) 
\end{equation}
If we plug the substitution \eqref{eqn:substitution} into the braid depicted above, and slide all the colons (except for the colon involving the red and green strands) to the far right of the picture, we conclude that the residue of $X$ is equal to:

\begin{figure}[ht]    
	\centering
	\includegraphics[scale=0.2]{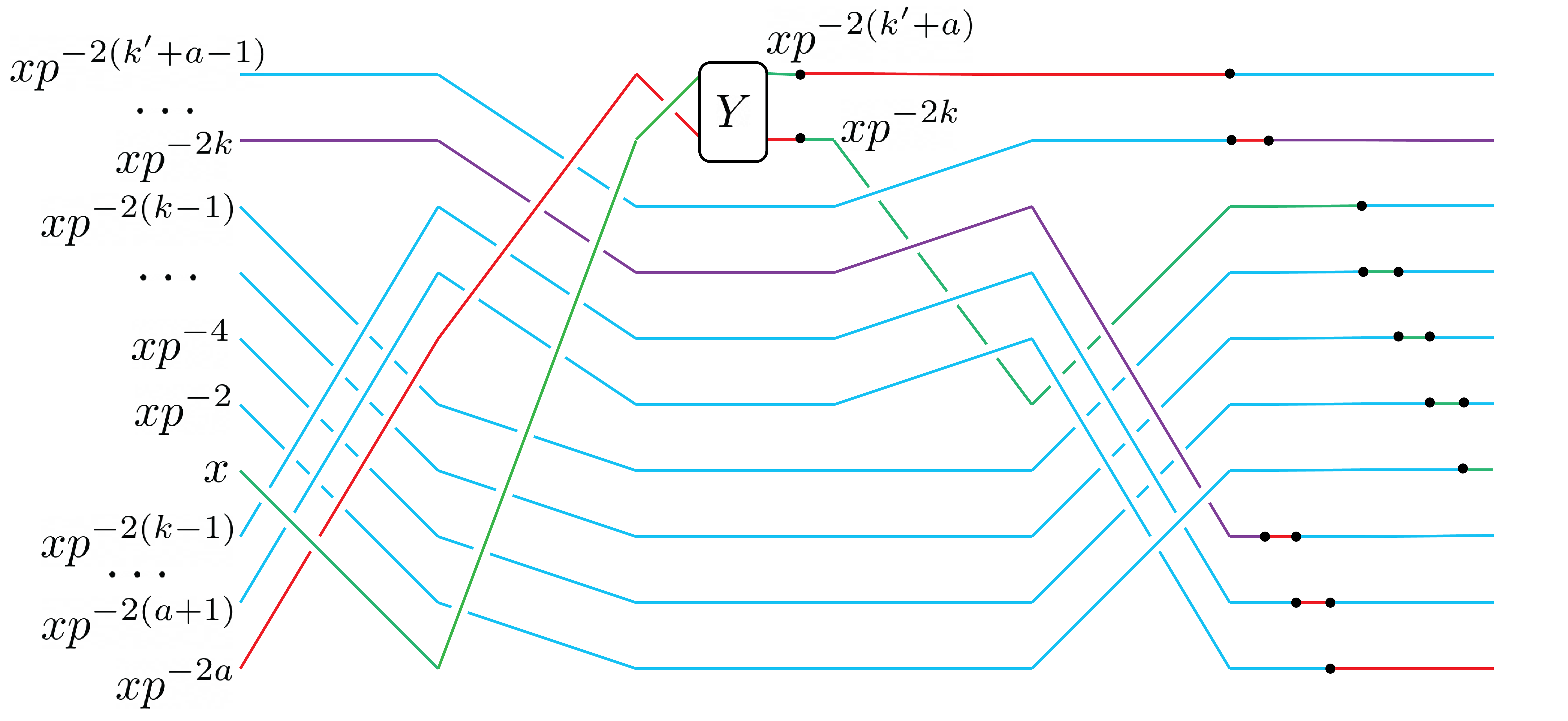}
\end{figure}

\noindent Note now the right-most intersection of the green and purple braid. Since the variables on these two braids are the same, formula \eqref{eqn:residues} for the residue of the $R$--matrix means that we can replace the crossing by a colon. Therefore, the braid above is equal to the one below (we apply some braid isotopies to simplify the left part of the picture): 

\begin{figure}[ht]    
	\centering
	\includegraphics[scale=0.2]{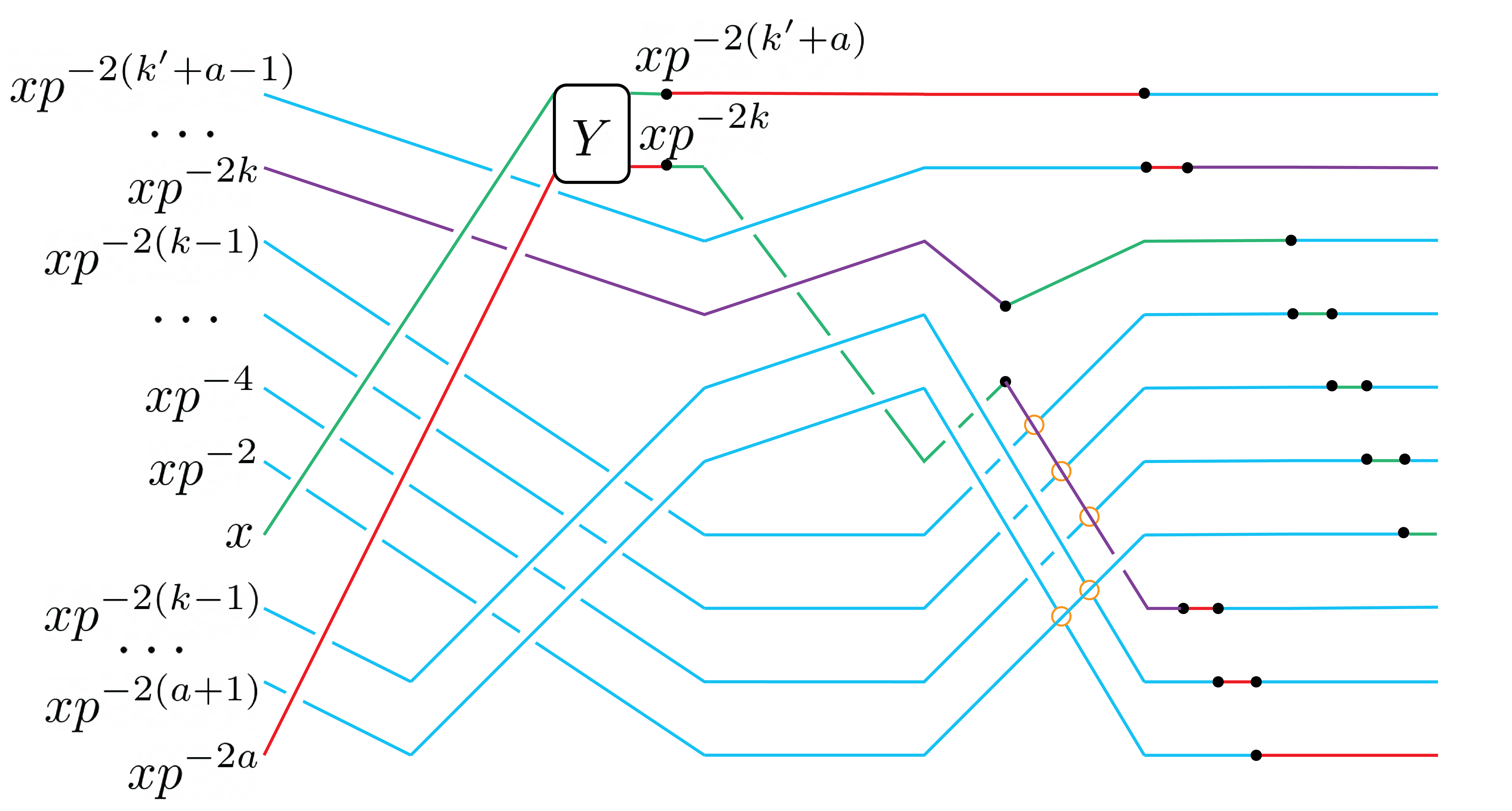}
\end{figure}

\noindent If we change the crossings marked by orange circles (which is allowed because of \eqref{eqn:unitary}), then we can transform the braid above into the braid below:

\begin{figure}[ht]    
	\centering
	\includegraphics[scale=0.2]{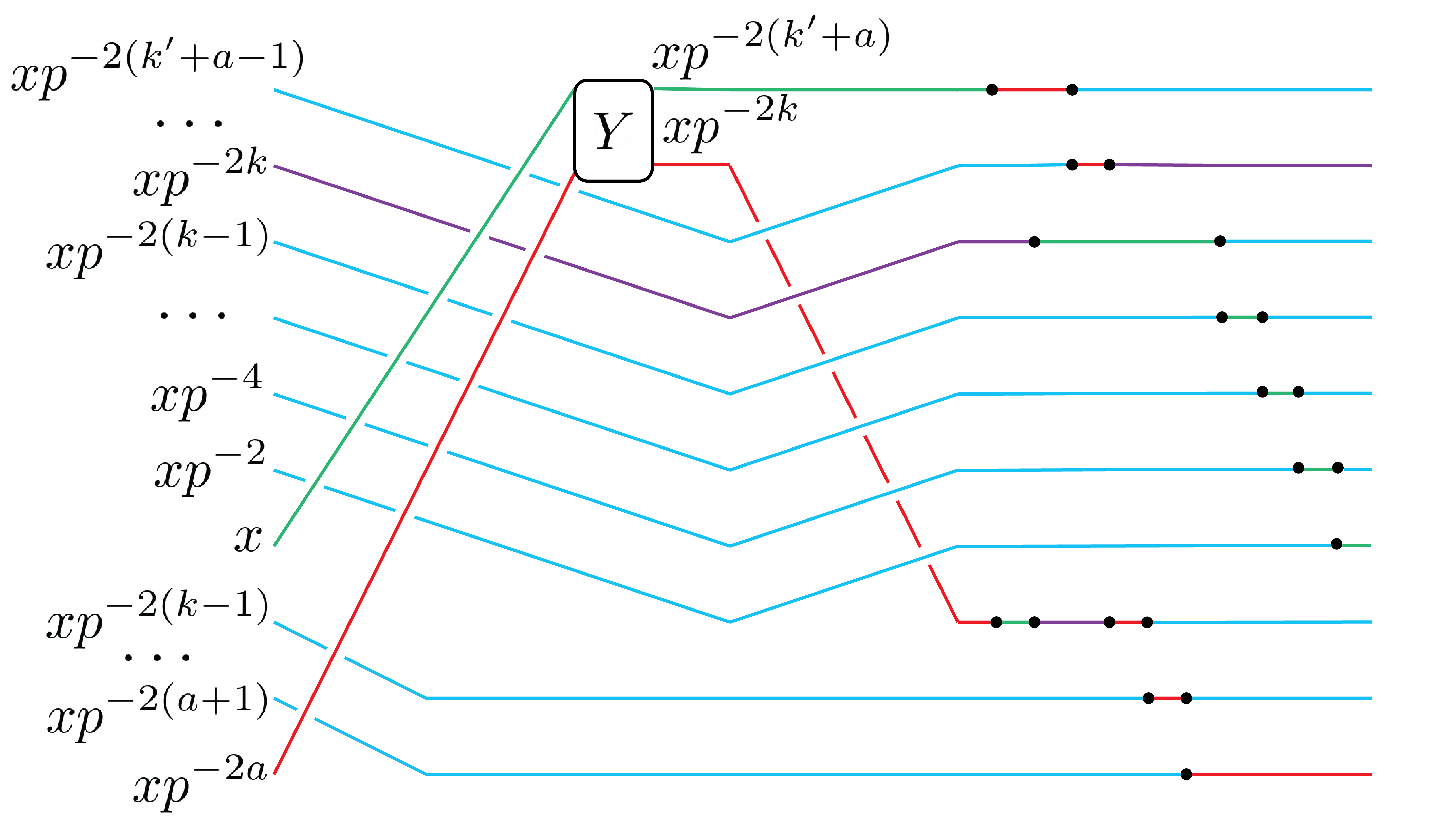}
\end{figure}

\noindent and conclude that $Y = X_{21}^{(k-a,k+a)}(xp^{-2a},x)$ times the cost of changing the crossings in question. This cost is precisely equal to:
$$
\frac {\prod_{i=a+1}^{k-1} f\left( \frac {p^2}{p^{2i}} \right)}{\prod_{i=2}^{k-1} f\left( \frac {p^{2k}}{p^{2i}} \right)} = \frac 1{\prod_{i=1}^{a-1} f(p^{-2i})}
$$
This implies formula \eqref{eqn:finally}, and with it \eqref{eqn:res 1}. Having proved \eqref{eqn:res 1} for all $k,k',a$ such that $k' + a > k > a > 0$, we note that \eqref{eqn:res 2} for all $k,k',a$ such that $k > k' + a > 0 > a$ follows from it, due to the equality of rational functions below:
\begin{equation}
\label{eqn:sym}
X^{(k,k')}(x,y) = 
\end{equation}
$$
R_{12} \left( \frac xy \right) X_{21}^{(k',k)}(y,x) R_{12} \left( \frac {xp^{2k'}}{yp^{2k}} \right)^{-1} \begin{cases} \prod_{i=1}^{k-k'} f \left( \frac x{y p^{2i}} \right) &\text{if } k \geq k' \\ \prod_{i=k-k'+1}^{0} f \left( \frac x{y p^{2i}} \right)^{-1} &\text{if } k < k' \end{cases}
$$
The $k=k'$ case of \eqref{eqn:sym} was proved in Proposition 4.12 of \cite{Tale}, and the general case is proved by an analogous argument. Alternatively, the interested reader could prove \eqref{eqn:res 2} directly by adapting the argument presented above.
	
\end{proof} 

\begin{proposition}

For any integers $k$ and $k'$, we have:
$$
\left \langle R_{12} \left(\frac {xp^{2k}}{yp^{2k'}} \right) W^{(k)}_1(x) R_{21} \left(\frac {yp^{2k'}}x \right) W^{(k')}_2(y), X^- \right \rangle \prod_{i=k'-k+1}^{k'-1} f \left(\frac x{y p^{2i}} \right) = 
$$
\begin{equation}
\label{eqn:expand 1}
= (q^2-1)^{k+k'} \cdot X^{(k,k')} \left( \frac 1x, \frac 1y \right)^{\dagger_1 \dagger_2}
\end{equation} 
expanded for $|x| \gg |y|$, and:
$$
\left \langle W^{(k')}_2(y) R_{12} \left(\frac {xp^{2k}}y \right) W^{(k)}_1(x) R_{21} \left(\frac {y}x \right) , X^- \right \rangle \prod_{i=1}^{k-1} f \left( \frac {y}{xp^{2i}} \right) =
$$
\begin{equation}
\label{eqn:expand 2}
= (q^2-1)^{k+k'} \cdot X^{(k,k')} \left( \frac 1x, \frac 1{y} \right)^{\dagger_1 \dagger_2}
\end{equation} 
expanded for $|x| \ll |y|$. \\

\end{proposition}

\begin{proof} Formulas \eqref{eqn:expand 1} and \eqref{eqn:expand 2} are equivalent in virtue of \eqref{eqn:sym}, so we only prove the former. First, we will give a proof of \eqref{eqn:expand 1} based on a ``wrong assumption", and then explain how the proof can be adapted in order to give a correct solution (so in other words, the wrong assumption is simply a linguistic tool for better explanation). The assumption we make is that the series \eqref{eqn:def w series} can be written as:
\begin{equation}
\label{eqn:wrong}
W_\circ^{(k)}(x) =\left [n(q^{-1}-q) \right]^{1-k} \cdot \sym \Big[(12)...(1k) A_{1\circ}
\end{equation}
$$
D_1^{-1}...D_k^{-1} R_{k1}(p^{2-2k})...R_{21} (p^{-2})  \cdot F(x|z_1,...,z_k) \Big]	
$$
where:
$$
A_{1\circ} = \sum_{1\leq i,j \leq n} E_{ij} \otimes E_{ij} \in \End(V \otimes V)
$$
with the two factors of $V \otimes V$ being indexed by the labels ``1" and ``$\circ$", respectively. Moreover (and this is the point which is strictly speaking wrong), we ask that $F(x|z_1,...,z_k)$ be a distribution which has the formal property that:
\begin{equation}
\label{eqn:cov}
\int_{|z_1|\ll ... \ll |z_k|} F(x|z_1,...,z_k) X(z_1,...,z_k) = \underset{\left\{z_1 = \frac 1x, z_2 = \frac 1{xp^2},..., z_k = \frac 1{xp^{2k-2}} \right\}}{\Res} X
\end{equation} 
With this convention, formally applying property \eqref{eqn:symm pair 1} would give:
$$
\langle W^{(k)}_\circ(x), X^- \rangle = (q^2-1)^k \left [n(q^{-1}-q) \right]^{1-k} \int_{|z_1|\ll ... \ll |z_k|} F(x|z_1,...,z_k)  \text{Tr}_{V^{\otimes k}} \Big[ (12) ... (1k)  
$$
\begin{equation}
\label{eqn:qat}
\left. A_{1\circ} D_1^{-1}...D_k^{-1} R_{k1}(p^{2-2k})...R_{21} (p^{-2})  \frac {X(z_1,...,z_k)}{\prod_{1 \leq i < j \leq k} f(z_i/z_j)} \right]
\end{equation}
where $\text{Tr}_{V^{\otimes k}}$ denotes the trace over the factors of $V$ with labels $1,...,k$. Then formula \eqref{eqn:cov} implies that \eqref{eqn:qat} is: 
$$
(q^2-1)^k \left[n(q^{-1}-q) \right]^{1-k} \text{Tr}_{V^{\otimes k}} \Big[  (12) ... (1k) A_{1\circ} D_1^{-1}...D_k^{-1} R_{k1}(p^{2-2k})...R_{21} (p^{-2}) 
$$
$$
\left. \underset{z_i = \frac 1{xp^{2i-2}} }{\Res} \frac {X}{\prod_{i<j} f \left( \frac {z_i}{z_j} \right)} \right] = (q^2-1)^k n^{1-k} \text{Tr}_{V^{\otimes k}} \left[ A_{1\circ} X_1^{(k)} \left(\frac 1x\right) \right] = (q^2-1)^k X_\circ^{(k)} \left(\frac 1x\right)^\dagger
$$
as expected from Proposition \ref{prop:pair w}. In fact, the latter computation is the only reason we make the ``wrong assumption" \eqref{eqn:wrong}. Formulas \eqref{eqn:symm pair 1} and \eqref{eqn:wrong} imply:
\begin{equation}
\label{eqn:lhs}
\left \langle R_{\circ \bullet} \left(\frac {xp^{2k}}{yp^{2k'}} \right) W^{(k)}_\circ(x) R_{\bullet \circ} \left(\frac {yp^{2k'}}x \right) W^{(k')}_\bullet(y), X^- \right \rangle \prod_{i=k'-k+1}^{k'-1} f \left(\frac x{y p^{2i}} \right) = 
\end{equation}
$$
= (q^2-1)^{k+k'} \int_{|z_1|\ll ... \ll |z_{k+k'}|} F(x|z_1,...,z_k)F(y|z_{k+1},...,z_{k+k'})   \text{Tr}_{V^{\otimes k+k'}} \left[ R_{\circ \bullet} \left(\frac {xp^{2k}}{yp^{2k'}} \right) \right.
$$
$$
 \prod_{i=k}^1 \prod_{j=k+1}^{k+k'} R_{ij} \left( \frac {z_i}{z_j} \right) \prod_{i=2}^k (1i) A_{1\circ} \prod_{i=k}^2 R_{i1}(p^{2-2i}) \prod_{i=1}^k D_i^{-1} R_{\bullet \circ} \left(\frac {yp^{2k'}}x \right) \prod_{i=1}^k \prod_{j=k+k'}^{k+1} \tR^+_{ij} \left(\frac {z_i}{z_j} \right) 
$$
$$
\prod_{i=2}^{k'} (k+1,k+i) A_{k+1,\bullet} \prod_{i=k'}^{2} R_{k+i,k+1}(p^{2-2i}) \prod_{i=1}^{k'} D_{k+i}^{-1}  \prod_{i=k'-k+1}^{k'-1} f \left(\frac x{y p^{2i}} \right) \left. \frac {X(z_1,...,z_{k+k'})}{\prod_{i < j} f \left( \frac {z_i}{z_j} \right)} \right] 
$$
The formal property \eqref{eqn:cov} implies that \eqref{eqn:lhs} equals the $|x| \gg |y|$ expansion of: 
\begin{equation}
\label{eqn:theta}
(q^2-1)^{k+k'} \left[n(q^{-1}-q) \right]^{2-k-k'} \text{Tr}_{V^{\otimes k+k'}} \left[ R_{\circ \bullet} \left(\frac {xp^{2k}}{yp^{2k'}} \right) \right. \prod_{i=k}^1 \prod_{j=1}^{k'} R_{i,k+j} \left( \frac {y p^{2j}}{x p^{2i}} \right) 
\end{equation}
$$
\prod_{i=2}^k (1i) A_{1\circ} \prod_{i=k}^2 R_{i1}(p^{2-2i}) \prod_{i=1}^k D_i^{-1} R_{\bullet \circ} \left(\frac {yp^{2k'}}x \right) \prod_{i=1}^k \prod_{j=k'}^{1} \tR^+_{i,k+j} \left(\frac {yp^{2j}}{xp^{2i}} \right) \prod_{i=2}^{k'} (k+1,k+i) 
$$
$$
A_{k+1,\bullet} \prod_{i=k'}^{2} R_{k+i,k+1}(p^{2-2i}) \prod_{i=1}^{k'} D_{k+i}^{-1}  \prod_{i=k'-k+1}^{k'-1} f \left(\frac x{y p^{2i}} \right) \left. \text{Res} \left( \frac {X(z_1,...,z_{k+k'})}{\prod_{1\leq i < j \leq k+k'} f \left( \frac {z_i}{z_j} \right)} \right) \right] =: \Theta
$$
where $\Res$ will henceforth denote the residue at the collection:
\begin{equation}
\label{eqn:big residue}
z_1 = \frac 1{x},...,z_k = \frac 1{xp^{2k-2}}, z_{k+1} = \frac 1{y},..., z_{k+k'} = \frac 1{yp^{2k'-2}}
\end{equation}
A hefty but straightforward computation, which we will do at the end of the proof, claims that:
\begin{equation}
\label{eqn:hefty}
\Theta = (q^2-1)^{k+k'} \cdot X^{(k,k')}_{\circ \bullet} \left( \frac 1x, \frac 1y \right)^{\dagger_\circ \dagger_\bullet}
\end{equation}
which concludes the proof of \eqref{eqn:expand 1}, subject to the wrong assumption \eqref{eqn:wrong}. Now let us explain why we do not actually need the wrong assumption, and adapt the argument above to give a correct proof of \eqref{eqn:expand 1}. The left-hand side of \eqref{eqn:lhs} is a sum of terms of the form:
\begin{equation}
\label{eqn:hefty 2}
\sum_{1\leq i,j,i',j' \leq n}^{d,d'} \frac {E^{(\circ)}_{ij}}{x^d} \otimes \frac {E^{(\bullet)}_{i'j'}}{{y}^{d'}} \sum_{(u,v), (u',v') \in \zzz} \gamma \cdot \Big \langle W_{uv}^{(k)} W_{u'v'}^{(k')}, X^- \Big \rangle
\end{equation}
for various coefficients $\gamma$. By \eqref{eqn:bialg 1}, the pairing in the formula above equals:
\begin{equation}
\label{eqn:hefty 3}
\Big \langle W_{u'v'}^{(k')} \otimes W_{uv}^{(k)} , \Delta(X^-) \Big \rangle = \Big \langle W_{u'v'}^{(k')} \otimes W_{uv}^{(k)} , \text{RHS of \eqref{eqn:coproduct}} \Big \rangle
\end{equation}
By definition, the RHS of \eqref{eqn:coproduct} is simply $X$ itself with its variables expanded in the range $|z_1|, ..., |z_k| \ll |z_{k+1}|,...,|z_{k+k'}|$ and multiplied with the various coefficients of the series $S^-$, $T^-$ on the left and on the right, respectively. Then we can compute \eqref{eqn:hefty 3} by simply invoking \eqref{eqn:pairing w} twice. It is straightforward to see that the answer is given by the expression \eqref{eqn:theta} (the various double products of $R$'s and $\tR^+$'s are produced when pairing with the coefficients of the series $S^-$, $T^-$, see \cite{Tale}). In other words, while computationally involved and almost impossible to write down compactly, it is algorithmically clear how to show that \eqref{eqn:hefty 2} equals \eqref{eqn:theta}. So we are left with proving \eqref{eqn:hefty}, which we will do by replacing the residue in the definition of $\Theta$ with the operator prescribed by Figure 4. We have:
$$
\Theta = (q^2-1)^{k+k'} n^{2-k-k'} \text{Tr}_{V^{\otimes k+k'}} \left[ R_{\circ \bullet} \left(\frac {xp^{2k}}{yp^{2k'}} \right) \right. \prod_{i=k}^1 \prod_{j=1}^{k'} R_{i,k+j} \left( \frac {y p^{2j}}{x p^{2i}} \right) \prod_{i=2}^k (1i) A_{1\circ} 
$$
$$
\prod_{i=k}^2 R_{i1}(p^{2-2i}) \prod_{i=1}^k D_i^{-1} R_{\bullet \circ} \left(\frac {yp^{2k'}}x \right) \prod_{i=1}^k \prod_{j=k'}^{1} \tR^+_{i,k+j} \left(\frac {yp^{2j}}{xp^{2i}} \right) \prod_{i=2}^{k'} (k+1,k+i) A_{k+1,\bullet}
$$
$$
\prod_{i=k'}^{2} R_{k+i,k+1}(p^{2-2i}) \prod_{i=1}^{k'} D_{k+i}^{-1} \prod_{i=1}^{k+k'} D_i \prod_{i=2}^k R_{1i}(p^{2i-2}) \prod_{i=2}^{k'} R_{k+1,k+i}(p^{2i-2}) \prod_{i=2}^{k'} R_{1,k+i} \left(\frac {yp^{2i-2}}x\right)
$$
$$
\left. \frac {X^{(k,k')}_{1,k+1} \left(\frac 1x, \frac 1y \right)  }{\prod_{i=1-k}^{k'-k} f\left(\frac {x}{yp^{2i}} \right) \prod_{i=1}^{k'-1} f(p^{2i}) \prod_{i=1}^{k-1} f(p^{2i})}  \prod_{i=k'}^2 R_{k+i,1} \left(\frac {xp^{2k}}{yp^{2i-2}}\right) \prod_{i=k}^2 (1i) \prod_{i=k'}^{2} (k+1,k+i)  \right]
$$	
By the cyclic property of the trace and formula \eqref{eqn:unitary}, the expression above equals:
$$
\Theta = (q^2-1)^{k+k'} n^{2-k-k'} \text{Tr}_{V^{\otimes k+k'}} \left[ \prod_{i=k'}^2 R_{k+i,1} \left(\frac {xp^{2k}}{yp^{2i-2}}\right) \prod_{i=k}^2 (1i) \right. 
$$
$$
\prod_{i=k'}^{2} (k+1,k+i) R_{\circ \bullet} \left(\frac {xp^{2k}}{yp^{2k'}} \right) \prod_{i=k}^1 \prod_{j=1}^{k'} R_{i,k+j} \left( \frac {y p^{2j}}{x p^{2i}} \right) \prod_{i=2}^k (1i) A_{1\circ} \squiggly{\prod_{i=k}^2 R_{i1}(p^{2-2i})} 
$$
$$
\prod_{i=1}^k D_i^{-1}  R_{\bullet \circ} \left(\frac {yp^{2k'}}x \right) \prod_{i=1}^k \prod_{j=k'}^{1} \tR^+_{i,k+j} \left(\frac {yp^{2j}}{xp^{2i}} \right) \prod_{i=2}^{k'} (k+1,k+i) A_{k+1,\bullet} \prod_{i=1}^{k} D_i 
$$
$$
\left. \squiggly{\prod_{i=2}^k R_{1i}(p^{2i-2})} \prod_{i=2}^{k'} R_{1,k+i} \left(\frac {yp^{2i-2}}x\right) \frac {X^{(k,k')}_{1,k+1} \left(\frac 1x, \frac 1y \right) }{\prod_{i=1 - k}^{k'-k} f\left(\frac {x}{yp^{2i}} \right) \prod_{i=1}^{k-1} f(p^{2i})}  \right]
$$
In the formula above, we can send the products with squiggly underlines toward each other, where they will mutually annihilate due to \eqref{eqn:unitary}. In the process, we will re-order some of the $\tR^+$ terms due to the Yang-Baxter equation \eqref{eqn:ybe} (and we also need to use the fact that $R_{ij}$ and $\tR_{ij}$ commute with the operator $D_iD_j$): 
$$
\Theta = (q^2-1)^{k+k'} n^{2-k-k'} \text{Tr}_{V^{\otimes k+k'}} \left[ \prod_{i=k'}^2 R_{k+i,1} \left(\frac {xp^{2k}}{yp^{2i-2}}\right) \squiggly{\prod_{i=k}^2 (1i)} \right. 
$$
$$
\squiggly{\prod_{i=k'}^{2} (k+1,k+i)} R_{\circ \bullet} \left(\frac {xp^{2k}}{yp^{2k'}} \right) \prod_{i=k}^1 \prod_{j=1}^{k'} R_{i,k+j} \left( \frac {y p^{2j}}{x p^{2i}} \right) \squiggly{\prod_{i=2}^k (1i)} A_{1\circ}  R_{\bullet \circ} \left(\frac {yp^{2k'}}x \right)
$$
$$
\prod_{j=k'}^{1} \prod_{i=2}^k \left[ D_i^{-1} \tR^+_{i,k+j} \left(\frac {yp^{2j}}{xp^{2i}} \right) D_i \right]  \prod_{j=k'}^{1} \left[ D_1^{-1} \tR^+_{1,k+j} \left(\frac {yp^{2j}}{xp^{2}} \right) D_1 \right] 
$$
$$
\left. \squiggly{\prod_{i=2}^{k'} (k+1,k+i)} A_{k+1,\bullet}\prod_{i=2}^{k'} R_{1,k+i} \left(\frac {yp^{2i-2}}x\right) \frac {X^{(k,k')}_{1,k+1} \left(\frac 1x, \frac 1y \right)}{\prod_{i=1-k}^{k'-k} f\left(\frac {x}{yp^{2i}} \right)} \right]
$$
Now we will send the permutations with squiggly underlines toward each other, where they will mutually annihilate. However, this process will change the indices on all tensors between squiggly lines, and we obtain:
$$
\Theta = (q^2-1)^{k+k'} n^{2-k-k'} \text{Tr}_{V^{\otimes k+k'}} \left[  R_{\circ \bullet} \left(\frac {xp^{2k}}{yp^{2k'}} \right) \prod_{i=k'}^2 R_{k+i,1} \left(\frac {xp^{2k}}{yp^{2i-2}}\right) \right. 
$$
$$
 \prod_{j=2}^{k'} R_{1,k+j} \left( \frac {y p^{2j-2}}{x p^{2k}} \right)  R_{1,k+1} \left( \frac {y p^{2k'}}{x p^{2k}} \right) \prod_{j=2}^{k'} \prod_{i=k}^2  R_{i,k+j} \left( \frac {y p^{2j-2}}{x p^{2i-2}} \right) \prod_{i=k}^2  R_{i,k+1} \left( \frac {y p^{2k'}}{x p^{2i-2}} \right) 
$$
$$
A_{1\circ}  R_{\bullet \circ} \left(\frac {yp^{2k'}}x \right) \prod_{i=2}^k \left[ D_i^{-1} \tR^+_{i,k+1} \left(\frac {yp^{2k'}}{xp^{2i}} \right) D_i \right] \prod_{j=k'}^{2} \prod_{i=2}^k \left[ D_i^{-1} \tR^+_{i,k+j} \left(\frac {yp^{2j-2}}{xp^{2i}} \right) D_i \right]
$$
$$
\left[ D_1^{-1} \tR^+_{1,k+1} \left(\frac {yp^{2k'}}{xp^{2}} \right) D_1 \right] \prod_{j=k'}^{2} \left[ D_1^{-1} \tR^+_{1,k+j} \left(\frac {yp^{2j-2}}{xp^{2}} \right) D_1 \right]  A_{k+1,\bullet} \prod_{i=2}^{k'} R_{1,k+i} \left(\frac {yp^{2i-2}}x\right)
$$
\begin{equation}
\label{eqn:theta 2}
\left.  \cdot \frac {X^{(k,k')}_{1,k+1} \left(\frac 1x, \frac 1y \right)}{\prod_{i=1-k}^{k'-k} f\left(\frac {x}{yp^{2i}} \right)} \right]
\end{equation}
Recall the explicit formulas \eqref{eqn:r} and \eqref{eqn:tr plus}, which imply the identity:
\begin{equation}
\label{eqn:property}
\left[ D_1^{-1} \tR^+_{12} (z) D_1 \right] ^{\dagger_1} R_{12} (zp^2)^{\dagger_1} = \text{Id}_{V \otimes V}
\end{equation}
where $\dagger_1$ is transposition in the first tensor factor. As a consequence, we have:
\begin{equation}
\label{eqn:elem 1} 
\text{Tr}_{V \otimes V} \left( A_{23...} R_{12}(zp^2) B_{13...} \Big[ D_1^{-1} \tR^+_{12}(z) D_1 \Big]\right) = \text{Tr}_{V \otimes V}(A_{23...} B_{13...})
\end{equation}
where the trace is only taken over the tensor factors 1 and 2. In other words, formula \eqref{eqn:elem 1} states that we can ``cancel" factors $R_{12}(zp^2)$ and $D_1^{-1} \tR^+_{12}(z) D_1$ from any trace, as long as there are no terms with index $2$ in the interval between these factors and no terms with index $1$ outside the interval between these factors. By applying formula \eqref{eqn:elem 1}, we can thus cancel factors in \eqref{eqn:theta} pairwise, and obtain: 
$$
\Theta = (q^2-1)^{k+k'} n^{2-k-k'}\text{Tr}_{V^{\otimes k+k'}} \left[  R_{\circ \bullet} \left(\frac {xp^{2k}}{yp^{2k'}} \right) R_{1,k+1} \left( \frac {y p^{2k'}}{x p^{2k}} \right) \right. 
$$
\begin{equation}
\label{eqn:theta 3}
A_{1\circ}  R_{\bullet \circ} \left(\frac {yp^{2k'}}x \right) \left[ D_1^{-1} \tR^+_{1,k+1} \left(\frac {yp^{2k'}}{xp^{2}} \right) D_1 \right]  A_{k+1,\bullet} \left. \frac {X^{(k,k')}_{1,k+1} \left(\frac 1x, \frac 1y \right)}{f \left( \frac {xp^{2k}}{yp^{2k'}} \right)} \right]
\end{equation}
As a consequence of \eqref{eqn:property}, we have the elementary identity:
\begin{equation}
\label{eqn:elem 2}
A_{1 \circ} R_{\circ \bullet}(zp^2)^{\dagger_\circ \dagger_\bullet} \Big[ D_1^{-1} \tR^+_{12}(z) D_1 \Big] A_{2\bullet} =  A_{1 \circ} A_{2\bullet}
\end{equation}
while as a consequence of \eqref{eqn:unitary}, we have the elementary identity:
\begin{equation}
\label{eqn:elem 3}
R_{\bullet \circ} \left(\frac 1z \right)^{\dagger_\circ \dagger_\bullet} R_{12} \left(z\right) A_{1 \circ} A_{2\bullet} = A_{1 \circ} A_{2\bullet} \cdot f(z)
\end{equation}
Plugging these two identities into \eqref{eqn:theta 3} yields precisely \eqref{eqn:hefty}. 	
	
\end{proof} 

\begin{proof} \emph{of Theorem \ref{thm:main 1}:} In order to show that:
\begin{equation}
\label{eqn:map}
\CP_n^\infty \rightarrow \wCA^+, \quad W_{ij}^{(k)} \rightarrow W_{ij}^{(k)}
\end{equation}
is an algebra homomorphism, we need to prove that the series \eqref{eqn:def w series} satisfy relations \eqref{eqn:w rels}. Because the pairing \eqref{eqn:pairing a} is non-degenerate, it suffices to show:
\begin{equation}
\label{eqn:hot}
\Big \langle \text{LHS of \eqref{eqn:w rels}}, X^- \Big \rangle (q^2-1)^{-k-k'} = \Big \langle \text{RHS of \eqref{eqn:w rels}}, X^- \Big \rangle (q^2-1)^{-k-k'}
\end{equation}
for any, henceforth fixed, $X^- \in \CA^-$. By \eqref{eqn:expand 1} and \eqref{eqn:expand 2}, the left-hand side of \eqref{eqn:hot} is equal to the difference between the expansions of the rational function:
\begin{equation}
\label{eqn:rat}
X^{(k,k')} \left( \frac 1x, \frac 1{y} \right)^{\dagger_1\dagger_2}
\end{equation}
at $|x|\gg|y|$ and at $|x|\ll|y|$. Therefore \eqref{eqn:identity} implies that the right-hand side of \eqref{eqn:hot} is equal to a sum over the poles of this rational function. More explicitly, Proposition \ref{prop:residues} implies that the right-hand side of \eqref{eqn:hot} is equal to:
$$
- \sum_{a \in \{-k',...,k-k'-1\} \sqcup \{1,...,k\}} \text{sgn}(a) \cdot \delta \left(\frac {x}{y} \cdot p^{2a} \right) \cdot \underset{y = x p^{2a}}{\text{Res}} \left[ X^{(k,k')} \left( \frac 1x, \frac 1{y} \right)^{\dagger_1\dagger_2} \right]
$$
Therefore, all that remains to show is that the residue above matches the pairing of the last row of \eqref{eqn:w rels} with $X^-$. Let us perform the computation for $a > 0$, as the case $a < 0$ is analogous, and we leave it to the interested reader. Formula:
$$
(q^{-1}-q) (12) \cdot \left \langle W^{(k'+a)}_1(x) R_{21} \left(p^{2k}\right) W^{(k-a)}_2(x p^{2a}) \prod_{i=1}^{k-1} f \left( p^{-2i} \right), X^- \right \rangle
$$
equals by \eqref{eqn:expand 2} (transposed, with indices swapped $1 \leftrightarrow 2$ and variables swapped $x \leftrightarrow y^{-1}$):
$$
(q^{-1} - q) (12)  X_{21}^{(k-a,k'+a)}\left( \frac 1{xp^{2a}}, \frac 1x \right)^{\dagger_1\dagger_2} R_{21} \left( p^{2a} \right) \prod_{i=1}^{a-1} f(p^{-2i})
$$
Due to \eqref{eqn:res 1}, the formula above is minus the residue of \eqref{eqn:rat} at $y = x p^{2a}$. This completes the proof of the fact that the map \eqref{eqn:map} is an algebra homomorphism. \\

\noindent Finally, we claim that the map \eqref{eqn:map} is an isomorphism because both $\CP_n^\infty $and $\wCA^+$ have linear bases (as completions) consisting of products:
\begin{equation}
\label{eqn:up}
W_{i_1j_1}^{(k_1)} ... W_{i_tj_t}^{(k_t)} \quad \text{where} \quad \frac {i_1-j_1}{k_1} \geq ... \geq \frac {i_t-j_t}{k_t}
\end{equation}
In the case of $\CP_n^\infty$, this is due to the definition \eqref{eqn:def p}, while in the case of $\wCA^+$, this is due to \eqref{eqn:linear basis 1} and the fact that the collection of elements \eqref{eqn:up} are upper triangular expressions in the collection of elements:
$$
F_{i_1j_1}^{(k_1)} ... F_{i_tj_t}^{(k_t)} \quad \text{where} \quad \frac {i_1-j_1}{k_1} \geq ... \geq \frac {i_t-j_t}{k_t}
$$
(this is a consequence of the definition \eqref{eqn:def w}, proved akin to formula (3.47) of \cite{W surf}). 

\end{proof}

\section{The second shuffle algebra}
\label{sec:shuf 2}

\subsection{} In the present Section, we will take \eqref{eqn:triangular 1} as the definition of $\UU$ (although we refer the reader to Definition 3.12 of \cite{Tale} for a review of the original definition in our conventions). Then \eqref{eqn:triangular 2} states that there exists an embedding:
\begin{equation}
\label{eqn:key 1}
\phi : \CA^+ \hookrightarrow \UU 
\end{equation}
In \cite{Par}, we defined a family of elements:
\begin{equation}
\label{eqn:key 2}
\tilde{W}_{ij}^{(k)} \in \wUU
\end{equation}
\footnote{The elements denoted by $W_{ij}^{(k)}$ in \loccit are our $\tilde{W}_{ij}^{(k)} (-1)^k p^{\frac {2(k-1)(j-i)+2(k-1)\bari+2\barj}n}$} where the completion $\wUU$ is such that it matches $\wCA^+$ under the map \eqref{eqn:key 1}. \\

\begin{proposition}
\label{prop:in fact}

We have $\phi \left( W_{ij}^{(k)} \right) = \tilde{W}_{ij}^{(k)}$ for all $(i,j) \in \zzz$ and $k \in \BN$. \\

\end{proposition}

\noindent The Proposition above will be proved at the end of the present Section, after we introduce the key notions that appear in \eqref{eqn:key 1} and \eqref{eqn:key 2}. \\

\begin{proof} \emph{of Theorem \ref{thm:main 2}:} By Theorem \ref{thm:main 1}, we have an algebra isomorphism
$$
\CP_n^\infty \cong \wCA^+
$$
where the elements denoted $W_{ij}^{(k)}$ in the LHS and the RHS correspond to each other. Therefore, we also have an isomorphism of the quotients:
$$
\CP_n^r \cong \wCA^+ \Big / \left(W_{ij} \right)_{(i,j) \in \zzz}^{k > r}
$$
Combining the isomorphism above with \eqref{eqn:key 1} and Proposition \ref{prop:in fact} yields:
$$
\CP_n^r \cong \phi ( \wCA^+ ) \Big / \left(\tilde{W}_{ij} \right)_{(i,j) \in \zzz}^{k > r}
$$
It was shown in Theorem 1.2 of \cite{Par} that the algebra in the RHS above acts on the $K$--theory groups that appear in the RHS of \eqref{eqn:k-theory}. Therefore, so does $\CP_n^r$.
	
\end{proof} 

\subsection{} 
\label{sub:shuf 1}

We will now define the shuffle algebras that appear in \eqref{eqn:triangular 1}. Consider variables $z_{ia}$ of color $i$, for various $i \in \{1,...,n\}$ and $a \in \BN$. We call a rational function:
$$
R(z_{11}...,z_{1d_1},...,z_{n1}...,z_{nd_n})
$$
color-symmetric if it is symmetric in the variables $z_{i1},...,z_{id_i}$ for all $i \in \{1,...,n\}$. Depending on the context, the symbol ``Sym" will refer to either color-symmetric functions in variables $z_{ia}$, or to the symmetrization operation:
$$
\sym \ F(...,z_{i1},...,z_{id_i},...) = \sum_{(\sigma^1,...,\sigma^n) \in S(d_1) \times ... \times S(d_n)} F(...,z_{i,\sigma^i(1)},...,z_{i,\sigma^i(d_i)},...)
$$
Let $\bd! = d_1!...d_n!$. The following construction arose in the context of quantum groups in \cite{E}, by analogy to the work of Feigin-Odesskii on certain elliptic algebras. \\

\begin{definition}
\label{def:shuf classic}
	
Consider the vector space:
\begin{equation}
\label{eqn:shuf classic}
\bigoplus_{\bd = (d_1,...,d_n) \in \nn} \fff(...,z_{i1},...,z_{id_i},...)^{\emph{Sym}}
\end{equation}
and endow it with an associative algebra structure, by setting $X*Y$ equal to:
$$
\esym \left[ \frac {X(...,z_{i1},...,z_{id_i},...)}{\bd!} \frac {Y(...,z_{i,d_i+1},...,z_{i,d_i+d'_i},...)}{\bd'!} \prod_{i,i' = 1}^n \prod^{1 \leq a \leq d_i}_{d_{i'} < a' \leq d_{i'}+d'_{i'}} \zeta \left( \frac {z_{ia}}{z_{i'a'}} \right) \right] 
$$
Let $\CS^\pm$ be the subset of \eqref{eqn:shuf classic} consisting of rational functions which have at most simple poles at $z_{ia}q^2 - z_{i+1,b}$ (for all applicable indices $i, a, b$), such that the residue at such a pole is divisible by $z_{ia'} - z_{i+1,b}$ and $z_{ia} - z_{i+1,b'}$ for all $a \neq a'$ and $b \neq b'$ (by convention, in the discussion above we set $z_{n+1,c} = z_{1c} \oq^{-2}$, $z_{0c} = z_{nc} \oq^2, \forall c$). \\
	
\end{definition}

\noindent Note that $\CS^+$ and $\CS^-$ are identical subalgebras of \eqref{eqn:shuf classic}. The reason why we make a distinction in notation is that we choose to grade these algebras differently, by:
\begin{equation}
\label{eqn:degree}
\deg R^\pm = (\pm \bd, k) \in \zz \times \BZ
\end{equation}
where $R^\pm$ is the element of $\CS^\pm$ corresponding to a rational function $R$ in the $\bd$--th direct summand of \eqref{eqn:shuf classic}, which has total homogeneous degree $k$. We will refer to the components of the degree vector as ``horizontal" and ``vertical" degree, respectively:
\begin{align*} 
&\hdeg R^\pm = \pm \bd \\
&\vdeg R^\pm = k
\end{align*} 
for any $R^\pm$ whose degree is of the form \eqref{eqn:degree}. \\

\subsection{} Let us define:
\begin{equation}
\label{eqn:cartan}
\CS^0 = \uui^{\otimes n} = \fff \Big \langle \psi_{i,k}^\pm, c, \barc \Big \rangle^{k \in \BN \sqcup 0}_{i \in \BZ} \Big/ \text{relations} 
\end{equation}
where the $\psi_{i,k}^\pm$ generate a rank $n$ deformed Heisenberg algebra with central charge $\barc$, and are connected by the following identifications, for all $i \in \BZ$ and $k \in \BN \sqcup 0$:
$$
\psi_{i,0}^+ = (\psi_{i,0}^-)^{-1} =: \psi_i \qquad \text{and} \qquad \psi_{i+n,k}^\pm = c^{\pm 1} \oq^{\mp 2k} \psi_{i,k}^\pm
$$
When $\barc = 1$, all the generators in \eqref{eqn:cartan} commute. For general $\barc$, the defining relations in \eqref{eqn:cartan} can be found in (3.38)--(3.40) of \cite{Tale}, where we used the notation:
$$
\ph^\pm_i(z) = \frac {\psi^\pm_{i+1}(zq^2)}{\psi_i^\pm (z)}
$$
where $\psi_i^\pm(z) = \sum_{k=0}^\infty \frac {\psi_{i,k}^\pm}{z^{\pm k}}$. \\

\begin{definition}
	
Consider the double shuffle algebra:
\begin{equation}
\label{eqn:whole}
\CS = \CS^+ \otimes \CS^0 \otimes \CS^{-,\eop} \Big / \text{relations}
\end{equation}
The specific relations we impose between the three factors of \eqref{eqn:whole} can be found in Subsection 3.14 of \cite{Tale}, but their precise form will not be important to us. \\

\end{definition}

\noindent In the present paper, we will take \eqref{eqn:whole} as the definition of $\UU$. However, using the standard definition of the quantum toroidal algebra of $\fgl_n$, the fact that it is isomorphic to $\CS$ as defined above is a theorem that can be found in \cite{Tor}. \\

\subsection{}
\label{sub:uuu}

For any choice of $(i<j) \in \zzz$ and $\mu \in \BQ$ such that $\frac {j-i}{\mu} \in \BZ$, we define:
\begin{align} 
&\bA^\mu_{\pm [i;j)} = \sym \left[ \frac {\prod_{a=i}^{j-1} (z_a \oq^{\frac {2a}n})^{\left \lfloor \frac {a-i+1}{\mu} \right \rfloor - \left \lfloor \frac {a-i}{\mu} \right \rfloor}}{(-q^2\oq^{\frac 2n})^{j-i}\left(1 - \frac {z_{i+1}}{z_i q^2}\right) ... \left(1 - \frac {z_{j-1}}{z_{j-2} q^2}\right)} \prod_{i\leq a < b < j} \zeta \left( \frac {z_b}{z_a} \right)  \right] \in \CS^\pm \label{eqn:shuf a} \\
&\bB^\mu_{\pm [i;j)} =\sym \left[ \frac {\prod_{a=i}^{j-1} (z_a \oq^{\frac {2a}n})^{\left \lceil \frac {a-i+1}{\mu} \right \rceil - \left \lceil \frac {a-i}{\mu} \right \rceil}}{(-\oq^{\frac 2n})^{i-j}\left(1 - \frac {z_i}{z_{i+1}}\right) ... \left(1 - \frac {z_{j-2}}{z_{j-1}}\right)} \prod_{i\leq a < b < j} \zeta \left( \frac {z_a}{z_b} \right)  \right] \in \CS^\pm \label{eqn:shuf b}
\end{align}
In order to think of the RHS of \eqref{eqn:shuf a} and \eqref{eqn:shuf b} as elements of the shuffle algebra, we relabel the variables $z_i,...,z_{j-1}$ according to the following rule $\forall a \in \{i,...,i+n-1\}$:
$$
z_a,z_{a+n},z_{a+2n}, ... \quad \leadsto \quad z_{\bar{a}1}, z_{\bar{a}2} \oq^{-2}, z_{\bar{a}3} \oq^{-4},...
$$
and thus the RHS of \eqref{eqn:shuf a} and \eqref{eqn:shuf b} are interpreted as elements of the vector space \eqref{eqn:shuf classic}. If $\frac {j-i}{\mu} \notin \BZ$, the LHS of \eqref{eqn:shuf a} and \eqref{eqn:shuf b} are defined to be 0. We will write:
$$
\bA_{\pm [i;j)}^{(\pm k)} = \bA_{\pm [i;j)}^\mu \qquad \text{and} \qquad  \bB_{\pm [i;j)}^{(\pm k)} = \bB_{\pm[i;j)}^\mu
$$
where $k = \frac {j-i}{\mu}$, in order to emphasize the fact that: 
$$
\deg \bA_{\pm [i;j)}^{(k)} = \deg \bB_{\pm [i;j)}^{(k)} = (\pm[i;j),k)
$$
Consider the following subalgebras of $\CS$: \\
	
\begin{itemize}
		
\item $\CS^\uparrow$ generated by $\bA_{\pm [i;j)}^{(k)}$ and $\psi_{i,k}^+$ for all $k>0$ and $(i < j) \in \zzz$ \\
		
\item $\CS^\infty$ generated by $\bA_{[i;j)}^{(0)}$ and $\bB_{-[i;j)}^{(0)}$ for all $(i<j) \in \zzz$ \\
		
\item $\CS^\downarrow$ generated by $\bB_{\pm [i;j)}^{(k)}$ and $\psi_{i,k}^-$ for all $k>0$ and $(i < j) \in \zzz$ \\	
		
\end{itemize}

\noindent It was shown in \cite{PBW} that the subalgebras $\CS^\uparrow$ and $\CS^\downarrow$ have linear bases indexed by products of the aforementioned generators, taken in order of $\frac {j-i}k$. Moreover, we have an isomorphism of vector spaces:
\begin{equation}
\label{eqn:pbw}
\CS\cong \CS^\uparrow \otimes \CS^\infty \otimes \CS^\downarrow
\end{equation}

\begin{theorem} (\cite{Tale}) There is an algebra isomorphism $\phi : \CA^+  \cong \CS^\uparrow$ given by:
\begin{align*} 
&F_{ij}^{(k)} \mapsto  \bA_{-[i;j)}^{(k)} \cdot ( - q^3 \oq^{\frac 2n} )^{j-i} (-p^{\frac 2n})^k \\ 
&F_{ii}^{(k)} \mapsto \frac {\psi_{i,k}^+}{\psi_i} \cdot \oq^{\frac {2ik}n} (-1)^k \\
&F_{ji}^{(k)} \mapsto \frac 1{\psi_j} \bA_{[i;j)}^{(k)} \psi_i \cdot (-p^{\frac 2n})^k
\end{align*}
for all $(i < j) \in \zzz$ and $k > 0$. Similarly, there are isomorphisms: 
\begin{equation}
\label{eqn:latter}
\CA^- \cong \CS^\downarrow \quad \text{and} \quad \uu \cong \CS^{\infty}
\end{equation}
With this in mind, \eqref{eqn:pbw} yields the triangular decomposition \eqref{eqn:triangular 2}. \\ \label{thm:two} \\ \end{theorem}

\subsection{} Consider the following elements of the shuffle algebra, defined akin to \eqref{eqn:shuf a}:
\begin{equation}
\label{eqn:new shuf}
A_{\pm [i;j),k} = \sym \left[ \frac {-(z_i \oq^{\frac {2i}n} )^k}{\left(1 - \frac {z_i q^2}{z_{i+1}}\right) ... \left(1 - \frac {z_{j-2} q^2}{z_{j-1}}\right)} \prod_{i\leq a < b < j} \zeta \left( \frac {z_b}{z_a} \right)  \right] \in \CS^\pm
\end{equation}
for all $k > 0$ and $(i<j) \in \zzz$. We also make the following conventions: 
$$
A_{\pm [i;j),0} = 0 \qquad \text{and} \qquad A_{\pm [i;i),k} = \delta_k^0 
$$
forall $i \in \BZ/n\BZ$, $k \geq 0$ and $(i<j) \in \zzz$. \\

\begin{proposition} (\cite{Par}) We have $A_{\pm [i;j),k} \in \CS^\uparrow$ for all applicable $(i,j)$ and $k \in \BN$. \\ \end{proposition}

\subsection{} Let $\wCS^\uparrow$ denote the completion which corresponds to $\wCA^+$ under the isomorphism $\phi$ of Theorem \ref{thm:two}, see (2.31) of \cite{Par}. We will consider the following elements:
$$
\tilde{W}_{ij}^{(k)} = \sum_{a+b+c = k}^{a, b , c \geq 0} \sum_{s \leq \min(i,j)} (-1)^k (-q^3 \oq^{\frac 2n})^{j-s} p^{\frac {2(k-1)(i-s)-2(k-1)\bari-2\barj}n}
$$
\begin{equation}
\label{eqn:tilde w}
\frac 1{\psi_i}  A_{[s;i),a} \psi_s \cdot \frac {\psi_{s,b}^+}{\psi_s} \oq^{\frac {2sb}n} \cdot A_{-[s;j),c}
\end{equation}
of $\widehat{\CS}^\uparrow$, for all $(i,j) \in \zzz$ and $k \in \BN$. Meanwhile, applying $\phi$ to \eqref{eqn:def w} yields:
$$
\phi \left( W_{ij}^{(k)} \right) = \sum_{k_1,...,k_t \in \BN}^{k_1+...+k_n = k} \sum^{i = c_0,c_1,...,c_{t-1},c_t = j}_{\frac {c_0-c_1}{k_1} > ... > \frac {c_{t-1}-c_t}{k_t}} \phi \left( F_{c_0c_1}^{(k_1)} \right) ... \phi \left(F_{c_{t-1}c_t}^{(k_t)} \right) p^{\frac {2(\alpha(v) - (k-1)\bari - \barj)}n}
$$
$$
= (-1)^k \sum^{k_1,...,k_u,...,k_t \in \BN}_{k_1+...+k_u+...+k_t = k} \sum^{i = c_0,c_1,...,c_{u-1} = c_{u},...,c_{t-1}, c_t = j}_{\frac {c_0-c_1}{k_1} > ... > \frac {c_{u-2}-c_{u-1}}{k_{u-1}}  > 0 > \frac {c_u-c_{u+1}}{k_{u+1}} > ... > \frac {c_{t-1}-c_t}{k_t}} 
$$
$$
p^{\frac {2(\alpha(v) - (k-1)\bari - \barj)}n} \cdot \frac 1{\psi_{i}} \bA_{[c_1;c_0)}^{(k_1)} ... \bA_{[c_{u-1};c_{u-2})}^{(k_{u-1})} \psi_{c_{u-1}} p^{\frac {2(k_1+...+k_{u-1})}n} \cdot 
$$
$$
\cdot \frac {\psi_{c_u,k_u}^+}{\psi_{c_u}} \oq^{\frac {2c_uk_u}n} \cdot \bA_{-[c_u;c_{u+1})}^{(k_{u+1})} ... \bA_{-[c_{t-1};c_t)}^{(k_t)} ( -q^3 \oq^{\frac 2n} )^{c_t-c_u} p^{\frac {2(k_{u+1}+...+k_t)}n}
$$
Letting $s = c_{u-1} = c_u$, $a = k_0+...+k_{u-1} $, $b=k_u$, $c = k_{u+1}+...+k_t$, we note that the formula above matches \eqref{eqn:tilde w}, thus proving Proposition \ref{prop:in fact}, once we establish: \\

\begin{proposition}

We have the following equality in $\CS^+$:
\begin{equation}
\label{eqn:eq plus}
\mathop{\sum^{j = c_0,c_1,...,c_{t-1}, c_t = i}_{\frac {c_0-c_1}{k_1} > ... > \frac {c_{t-1}-c_t}{k_t}  > 0}}^{t,k_1,...,k_t \in \BN}_{k_1+...+k_t = k} p^{\frac {2\alpha(v) -2(k-1)(j-i)+2k}n} \bA_{[c_1;c_0)}^{(k_1)} ... \bA_{[c_t;c_{t-1})}^{(k_t)}  = A_{[i;j),k}
\end{equation}
as well as the following equality in $\CS^-$:
\begin{equation} 
\label{eqn:eq minus}
\mathop{\sum^{i = c_0,c_1,...,c_{t-1}, c_t = j}_{0 > \frac {c_0-c_1}{k_1} > ... > \frac {c_{t-1}-c_t}{k_t}}}^{t,k_1,...,k_t \in \BN}_{k_1+...+k_t = k} p^{\frac {2\alpha(v)+2k}n} \bA_{-[c_0;c_1)}^{(k_1)} ... \bA_{-[c_{t-1};c_{t})}^{(k_t)} =  A_{-[i;j),k} 
\end{equation}
for all $k\in \BN$ and $(i < j) \in \zzz$. \\

\end{proposition}

\begin{proof} We will only take care of the $\pm = +$ case, as the $\pm = -$ case is analogous, and left to the interested reader. According to formula \eqref{eqn:shuf a}, we have:
\begin{equation}
\label{eqn:lhs rhs}
\text{LHS of \eqref{eqn:eq minus}} = \sym \left[ \frac {X(z_i,...,z_{j-1})}{\left(1 - \frac {z_{i+1}}{z_i q^2}\right) ... \left(1 - \frac {z_{j-1}}{z_{j-2} q^2}\right)} \prod_{i\leq a < b < j} \zeta \left( \frac {z_b}{z_a} \right)  \right] \in \CS^-
\end{equation}
where:
$$
X(z_i,...,z_{j-1}) = \sum^{k_1,...,k_t \in \BN}_{k_1+...+k_t = k} \sum^{j = c_0,c_1,...,c_{t-1}, c_t = i}_{\frac {c_0-c_1}{k_1} > ... > \frac {c_{t-1}-c_t}{k_t} > 0} p^{\frac {2\alpha(v)-2(k-1)(j-i)+2k}n} (-q^2\oq^{\frac 2n})^{i-j}
$$
$$
\prod_{s=1}^t \prod_{a=c_s}^{c_{s-1}-1} (z_a \oq^{\frac {2a}n})^{\left \lfloor \frac {a-c_s+1}{c_{s-1}-c_{s}}  k_s \right \rfloor - \left \lfloor \frac {a-c_s}{c_{s-1}-c_s}  k_s \right \rfloor} \prod_{s=1}^{t-1} \left(1 - \frac {z_{c_s}}{z_{c_s-1}q^2} \right)
$$
Then \eqref{eqn:eq plus} follows from the fact that $p = q^n \oq$ and the identity:
\begin{equation}
\label{eqn:of}
X(z_i,...,z_{j-1}) = (z_i \oq^{\frac {2i}n} )^{k} \frac {-z_{j-1}}{z_i(-q^2)^{j-i-1}} 
\end{equation}
If we perform the substitutions:
$$
c_{s-1}-c_s \leadsto d_s, \quad z_{c_s-\bullet} \leadsto q^{2(c_s-\bullet)} w_{d_1+...+d_s+\bullet}
$$
then \eqref{eqn:of} is an immediate consequence of formula (7.26) of \cite{W} and the following straightforward combinatorial identity:
$$
\sum_{s=1}^t \sum_{a=c_s}^{c_{s-1}-1} a \left( \left \lfloor \frac {a-c_s+1}{c_s-c_{s-1}} k_s \right \rfloor - \left \lfloor \frac {a-c_s}{c_s-c_{s-1}} k_s \right \rfloor \right) =  \sum_{s=1}^t \left( c_{s-1}k_s - \sum_{i=1}^{c_{s-1}-c_s} \left \lfloor \frac {ik_s}{c_{s-1}-c_s} \right \rfloor \right)
$$
$$
= \sum_{s=1}^t \left( c_{s-1}k_s - \frac {k_s(c_{s-1} - c_s) + k_s - (c_{s-1}-c_s) + \gcd(k_s,c_s-c_{s-1})}2 \right) = k j - k - \alpha(v)
$$

\end{proof}

\end{document}